\documentclass[11pt,a4paper]{article}

\usepackage{amsmath,amssymb,amsthm}
\usepackage[font=small,labelfont=bf]{caption}
\usepackage[font=small,labelfont=normalfont]{subcaption}
\usepackage{a4wide}
\usepackage{hyperref}
\usepackage{graphicx}
\usepackage{url}
\usepackage[mathlines]{lineno}
\usepackage{xcolor}
\usepackage{enumerate}

\usepackage{todonotes}

\hypersetup{pdfpagemode=UseNone, pdfstartview=}

\graphicspath{{figures_arXiv/}}
\newtheorem{thm}{Theorem}
\newtheorem*{thm*}{Theorem}
\newtheorem{prop}[thm]{Proposition}
\newtheorem{cor}[thm]{Corollary}
\newtheorem{obs}[thm]{Observation}
\newtheorem{lem}[thm]{Lemma}

\newtheorem{claim}{Claim}[thm]

\DeclareMathOperator{\conv}{conv}
\DeclareMathOperator{\cl}{cl}
\newcommand{\dline}[1]{\overline{#1}}
\newcommand{\astar}{a^*}

\title{A superlinear lower bound on the number of 5-holes}

\author{Oswin Aichholzer\thanks{Institute for Software Technology,
	Graz University of Technology, Austria,  {\tt [oaich,thackl,iparada,\allowbreak bvogt]@ist.tugraz.at} }
\and
Martin Balko\thanks{Department of Applied Mathematics and Institute for Theoretical Computer Science, Faculty of Mathematics and Physics, Charles University,
		Czech Republic, {\tt [balko,kyncl,valtr]@kam.mff.cuni.cz}} \thanks{Alfr\'{e}d R\'{e}nyi Institute of Mathematics, Hungarian Academy of Sciences, Budapest, Hungary}
\and
Thomas Hackl\textsuperscript{$\ast$}
\and
Jan Kyn\v{c}l\textsuperscript{$\dagger$}
\and 
Irene Parada\textsuperscript{$\ast$}
\and 
Manfred Scheucher\textsuperscript{$\ast$}\textsuperscript{$\ddagger$}\thanks{Institut f\"ur Mathematik, 
Technische Universit\"at Berlin, Germany,  {\tt [scheucher]@math.tu-berlin.de} }
\and
Pavel Valtr\textsuperscript{$\dagger$}\textsuperscript{$\ddagger$}
\and
Birgit Vogtenhuber\textsuperscript{$\ast$}
}

\begin{document}

\maketitle

\begin{abstract}
Let $P$ be a finite set of points in the plane in \emph{general position}, 
that is, no three points of $P$ are on a common line.
We say that a set $H$ of five points from $P$ is a \emph{$5$-hole in~$P$} 
if $H$ is the vertex set of a convex $5$-gon containing no other points of~$P$.
For a positive integer~$n$, let $h_5(n)$ be the minimum number of 5-holes 
among all sets of $n$ points in the plane in general position.

Despite many efforts in the last 30 years, the best known asymptotic lower and upper bounds for $h_5(n)$ have been of order $\Omega(n)$ and~$O(n^2)$, respectively.
We show that $h_5(n) = \Omega(n\log^{4/5}{n})$, obtaining the first superlinear lower bound on $h_5(n)$.

The following structural result, which might be of independent interest, is a crucial step in the proof of this lower bound.
If a finite set $P$ of points in the plane in general position is
partitioned by a line $\ell$ into two subsets, 
each of size at least 5 and not in convex position, 
then $\ell$ intersects the convex hull of some 5-hole in~$P$.
The proof of this result is computer-assisted.
\end{abstract}

%========================================================================================================
\section{Introduction}
We say that a set of points in the plane is in \emph{general position} if
it contains no three points on a common line.
A point set is in \emph{convex position} if it is the vertex set of a convex polygon.
In 1935, Erd\H{o}s and Szekeres~\cite{ES_1935} proved the following theorem, which is a classical result
both in combinatorial geometry and Ramsey theory.

\begin{thm*}[\cite{ES_1935}, The Erd\H{o}s--Szekeres Theorem]\label{thm:Erd-Sz}
For every integer $k\ge3$, there is a smallest integer $n=n(k)$ such that every set of at least
$n$ points in general position in the plane contains $k$ points in convex position.
\end{thm*}

The Erd\H{o}s--Szekeres Theorem motivated a lot of further research, including numerous modifications and
extensions of the theorem. Here we mention only results closely related to the main
topic of our paper.

Let $P$ be a finite set of points in general position in the plane.
We say that a set $H$ of $k$ points from $P$ is a \emph{$k$-hole in~$P$} if $H$ is the vertex set of a convex $k$-gon containing no other points of~$P$.
In the 1970s, Erd\H{o}s~\cite{Erdos1978} asked whether, for every positive integer $k$, there is a $k$-hole in every sufficiently large finite point set in general position in the plane.
Harborth~\cite{Ha78} proved that there is a $5$-hole in every set of $10$ points in general position in the plane
and gave a construction of $9$ points in general position with no $5$-hole.
After unsuccessful attempts of researchers to answer Erd\H{o}s' question affirmatively for any fixed integer $k \ge 6$,
Horton~\cite{Ho83} constructed, for every positive integer~$n$, a set of $n$ points in general position in the plane with no $7$-hole.
His construction was later generalized to so-called \emph{Horton sets}
and \emph{squared Horton sets}~\cite{Valtr1992} and to higher dimensions~\cite{VALTR1992115}.
The question whether there is a 6-hole in every sufficiently large finite planar point set remained open
until 2007 when Gerken~\cite{Gerken2008} and Nicol\'as~\cite{Nicolas2007} independently gave an affirmative answer. 

For positive integers $n$ and $k$, let $h_k(n)$ be the minimum number of $k$-holes in a set of $n$ points in general position in the plane.
Due to Horton's construction,  $h_k(n)=0$ for every $n$ and every $k\ge 7$.
Asymptotically tight estimates for the functions $h_3(n)$ and $h_4(n)$ are known.
The best known lower bounds are due to Aichholzer et al.~\cite{afhhpv-lbnsc-14} 
who showed that $h_3(n) \ge n^2 - \frac{32n}{7} + \frac{22}{7}$ and $h_4(n) \ge \frac{n^2}{2} - \frac{9n}{4} -o(n)$.
The best known upper bounds $h_3(n) \le 1.6196n^2+o(n^2)$ and $h_4(n) \le 1.9397n^2+o(n^2)$ 
are due to B\'{a}r\'{a}ny and Valtr~\cite{BV2004}.

For $h_5(n)$ and $h_6(n)$, no matching bounds are known.
So far, the best known asymptotic upper bounds on $h_5(n)$ and $h_6(n)$ 
were obtained by B\'{a}r\'{a}ny and Valtr~\cite{BV2004} and give $h_5(n) \le 1.0207n^2+o(n^2)$ and $h_6(n) \leq 0.2006n^2+o(n^2)$.
For the lower bound on $h_6(n)$, Valtr~\cite{v-ephpp-12} showed $h_6(n) \ge n/229 - 4$.

In this paper we give a new lower bound on $h_5(n)$.
It is widely conjectured that $h_5(n)$ grows quadratically in~$n$, 
but to this date only lower bounds on $h_5(n)$ that are linear in $n$  have been known.
As noted by B\'{a}r\'{a}ny and F\"{u}redi~\cite{BF}, a linear lower bound of $\lfloor n/10 \rfloor$ follows directly from Harborth's result~\cite{Ha78}.
B\'{a}r\'{a}ny and K\'{a}rolyi~\cite{BK} improved this bound to $h_5(n) \ge n/6 - O(1)$.
In 1987, Dehnhardt~\cite{De87} showed $h_5(11)=2$ and $h_5(12)=3$, 
obtaining $h_5(n) \ge 3 \lfloor n/12 \rfloor$.
However, his result remained unknown to the scientific community until recently.
Garc\'{i}a~\cite{g-nnet-12} then presented a proof of the lower bound $h_5(n) \ge 3 \lfloor \frac{n-4}{8} \rfloor$ 
and a slightly better estimate $h_5(n) \ge \lceil 3/7(n-11)\rceil$ 
was shown by Aichholzer, Hackl, and Vogtenhuber~\cite{ahv-5g5h-12}.
Quite recently, Valtr~\cite{v-ephpp-12} obtained $h_5(n) \geq n/2-O(1)$.
This was strengthened by Aichholzer et al.~\cite{afhhpv-lbnsc-14} to $h_5(n) \ge 3n/4-o(n)$.
All improvements on the multiplicative constant were achieved
by utilizing the values of $h_5(10)$, $h_5(11)$, and $h_5(12)$.
In the bachelor's thesis of Scheucher~\cite{scheucher2013}
the exact values $h_5(13)=3$, $h_5(14)=6$, and $h_5(15)=9$ were determined
and $h_5(16) \in \{10,11\}$ was shown. 
During the preparation of this paper, we further determined the value $h_5(16)=11$; 
see the webpage~\cite{program_manfred}.
Table~\ref{tab:5holes} summarizes our knowledge on the values of $h_5(n)$ for~$n \le 20$.
The values $h_5(n)$ for $n \le 16$ can be used to obtain further improvements on the multiplicative constant.
By revising the proofs of \cite[Lemma~1]{afhhpv-lbnsc-14} and \cite[Theorem~3]{afhhpv-lbnsc-14},
one can obtain $h_5(n) \ge n-10$ and $h_5(n) \ge 3n/2-o(n)$, respectively.
We also note that
it was shown 
in~\cite{PINCHASI2006385}
that if $h_3(n) \ge (1+\epsilon)n^2 - o(n^2)$, then $h_5(n) = \Omega(n^2)$.

\begin{table}[htb]
\centering
\begin{tabular}{r|rrrrrrrrrrrr}
$n$	&9	&10	&11	&12	&13	&14	&15	&16	&17	&18	&19	&20\\
\hline
%$h_5(n)$&0	&1	&2	&3	&3	&6	&9	&11	&11..16	&11..21	&11..26	&11..33\\
$h_5(n)$&0	&1	&2	&3	&3	&6	&9	&11	& $\leq 16$ & $\leq 21$ & $\leq 26$ & $\leq 33$\\
\end{tabular}
\caption{The minimum number $h_5(n)$ of 5-holes determined by any set of $n\le 20$ points.}
\label{tab:5holes}
\end{table}

As our main result, we give the first superlinear lower bound on $h_5(n)$. 
This solves an open problem, which was explicitely
stated, for example, in a book by Brass, Moser, and Pach~\cite[Chapter~8.4, Problem~5]{BMP05} and in the survey~\cite{aich09}.

\begin{thm}\label{theorem:theorem1}
There is an absolute constant $c > 0$ such that
for every integer~$n \ge 10$
we have $h_5(n) \ge c n \log^{4/5}n$.
\end{thm}

Let $P$ be a finite set of points in the plane in general position and let $\ell$ be a line that 
contains no point of~$P$.
We say that $P$ is \emph{$\ell$-divided} if there is at least one point of $P$ in each of the two halfplanes determined by~$\ell$.
For an $\ell$-divided set~$P$, we use $P=A \cup B$ to denote the fact that $\ell$ partitions $P$ into the subsets $A$ and~$B$.
In the rest of the paper,
we assume without loss of generality that $\ell$ is vertical and directed upwards, 
$A$ is to the left of~$\ell$, and $B$ is to the right of~$\ell$.

The following result, which might be of independent interest, is a crucial step in the proof of Theorem~\ref{theorem:theorem1}.

\begin{thm}\label{theorem:theorem2}
Let $P=A \cup B$ be an $\ell$-divided set with $|A|,|B| \ge 5$ and with neither $A$ nor $B$ in convex position.
Then there is an $\ell$-divided $5$-hole in~$P$.
\end{thm}

The proof of Theorem~\ref{theorem:theorem2} is computer-assisted. 
We reduce the result to several statements about point sets of size at most 11 and then verify each of these statements by an exhaustive computer search.
To verify the computer-aided proofs we have implemented two independent programs, which, in addition, are based on different abstractions of point sets; see Subsection~\ref{sec:section_computer}.
Some of the tools that we use originate from the bachelor's theses of Scheucher~\cite{scheucher2013,scheucher2014}.

Using a result of Garc\'ia~\cite{g-nnet-12}, we adapt the proof of Theorem~\ref{theorem:theorem1}
to provide improved lower bounds on the minimum numbers of 3-holes and 4-holes.

\begin{thm}\label{theorem:34holes}
The following two bounds are satisfied for every positive integer~$n$:
\begin{enumerate}[(i)]
 \item\label{thm:h_34:item1} $h_3(n) \ge n^2+\Omega(n \log^{2/3} n)$ and
 \item\label{thm:h_34:item2} $h_4(n) \ge \frac{n^2}{2}+\Omega(n \log^{3/4} n)$.
\end{enumerate}
\end{thm}

In the rest of the paper, we assume that every point set $P$ is planar, finite, and in general position. 
We also assume, without loss of generality, that all points in $P$ have distinct \mbox{$x$-coordinates}.
We use $\conv(P)$ to denote the convex hull of~$P$ and $\partial \conv(P)$ to denote the boundary of the convex hull of~$P$.

A subset $Q$ of $P$ that satisfies $P \cap \conv(Q) = Q$ is called
an \emph{island of~$P$}. 
Note that every $k$-hole in an island $Q$ of~$P$ 
is also a $k$-hole in~$P$. 
For any subset $R$ of the plane, 
if $R$ contains no point of~$P$, then we say that $R$ is \emph{empty of points of~$P$}.

In Section~\ref{section_theorem1} we derive quite easily Theorem~\ref{theorem:theorem1} from Theorem~\ref{theorem:theorem2}.
Theorem~\ref{theorem:34holes} is proved in Section~\ref{section_theorem34holes}.
Then, in Section~\ref{section_preliminaries}, we give some preliminaries for the proof of Theorem~\ref{theorem:theorem2}, 
which is presented in Section~\ref{section_theorem2}.
Finally, in Section~\ref{section_final_remarks}, we give some final remarks.
In particular, we show that the assumptions in Theorem~\ref{theorem:theorem2} are necessary.
To provide a better general view, 
we present a flow summary of the proof of Theorem~\ref{theorem:theorem1} in Appendix~\ref{appendix:flow_summary}.

%========================================================================================================
\section{Proof of Theorem~\ref{theorem:theorem1}} 
\label{section_theorem1}

We now apply Theorem~\ref{theorem:theorem2} to obtain a superlinear lower bound on the number of $5$-holes in a given set of $n$ points.
It clearly suffices to prove the statement 
for the case in which $n=2^t$ for some integer $t \geq 5^5$.

We prove by induction on $t \ge 5^5$ that the number of $5$-holes in an arbitrary set $P$ of $n=2^t$ points is at least 
$f(t) \mathrel{\mathop:}= c \cdot 2^t t^{4/5} = c \cdot n \log_2^{4/5}{n}$ 
for some absolute constant $c>0$.
For $t=5^5$, we have  $n>10$ and, by the result of Harborth~\cite{Ha78}, there is at least one $5$-hole in~$P$.
If the constant $c$ is sufficiently small, then  $f(t)=c \cdot n \log_2^{4/5}{n}
\leq 1$ and we have at least $f(t)$ $5$-holes in~$P$, which constitutes our base case.

For the inductive step we assume that $t > 5^5$.
We first partition $P$ with a line $\ell$ into two sets $A$ and $B$ of size $n/2$ each.
Then we further partition $A$ and $B$ into smaller sets using the following well-known lemma, which is, for example, implied by a result of Steiger and Zhao~{\cite[Theorem~1]{steigerZhao2010}}.

\begin{lem}[\cite{steigerZhao2010}]
\label{lemma:lemma3}
Let $P' = A' \cup B'$ be an $\ell$-divided set and let $r$ be a positive integer such that $r \leq |A'|, |B'|$.
Then there is a line that is disjoint from $P'$ and that determines an open halfplane $h$ with $|A' \cap h|=r=|B'  \cap h|$.
\end{lem}

We set $r \mathrel{\mathop:}= \lfloor \log_2^{1/5}{n}\rfloor$, $s \mathrel{\mathop:}= \lfloor n/(2r)\rfloor$, and apply Lemma~\ref{lemma:lemma3} iteratively in the following way to partition $P$ into islands $P_1,\dots,P_{s+1}$ of $P$ so that for every $i \in \{1,\dots,s\}$, the size of each $P_i \cap A$ and $P_i \cap B$ is exactly~$r$.
Let $P'_0 \mathrel{\mathop:}= P$.
For every $i=1,\dots,s$, we consider a line that is disjoint from $P'_{i-1}$ and that determines an open halfplane $h$ with $|P'_{i-1} \cap A \cap h|=r=|P'_{i-1} \cap B \cap h|$.
Such a line exists by Lemma~\ref{lemma:lemma3} applied to the $\ell$-divided set $P'_{i-1}$.
We then set 
$P_i \mathrel{\mathop:}= P'_{i-1} \cap h$,  $P'_i \mathrel{\mathop:}= P'_{i-1} \setminus P_i$, 
and continue with $i+1$.
Finally, we set $P_{s+1} \mathrel{\mathop:}= P'_s$.

Let $i \in \{1,\dots,s\}$.
If one of the sets $P_i \cap A$ and $P_i \cap B$ is in convex position, 
then there are at least $\binom{r}{5}$ $5$-holes in $P_i$ and, 
since $P_i$ is an island of~$P$, we have at least $\binom{r}{5}$ $5$-holes in~$P$.
If this is the case for at least $s/2$ islands $P_i$, 
then, given that $s = \lfloor n/(2r) \rfloor$ and thus $s/2 \ge \lfloor n/(4r) \rfloor$, we obtain at least $\lfloor n/(4r) \rfloor \binom{r}{5} \geq c \cdot n\log_2^{4/5}{n}$ $5$-holes in~$P$ for a sufficiently small constant $c>0$.

We thus further assume that for more than $s/2$ islands $P_i$, neither of the sets $P_i \cap A$ nor $P_i \cap B$ is in convex position.
Since $r = \lfloor \log_2^{1/5}{n}\rfloor \ge 5$, Theorem~\ref{theorem:theorem2} implies that there is an $\ell$-divided $5$-hole in each such~$P_i$.
Thus there is an $\ell$-divided $5$-hole in $P_i$ for more than $s/2$ islands $P_i$.
Since each $P_i$ is an island of~$P$ and since $s = \lfloor n/(2r)\rfloor$, 
we have more than $s/2 \ge \lfloor n/(4r)\rfloor$ $\ell$-divided 5-holes in~$P$.
As $|A|=|B|=n/2=2^{t-1}$, there are at least $f(t-1)$ $5$-holes in $A$ and at least $f(t-1)$ $5$-holes in~$B$ by the inductive assumption.
Since $A$ and $B$ are separated by the line~$\ell$, we have at least
\[\label{eq:lowerBound}
2f(t-1) +  n/(4r) = 2c (n/2)\log_2^{4/5}{(n/2)} +  n/(4r)
\ge c  n(t-1)^{4/5} +  n/(4t^{1/5})\]
$5$-holes in~$P$.
The right side of the above expression is at least $f(t) = c nt^{4/5}$, because
the inequality $cn(t-1)^{4/5}+ n/(4t^{1/5}) \geq c n t^{4/5}$ is equivalent to the inequality $(t-1)^{4/5}t^{1/5}+1/(4c) \geq  t$, 
which is true if the constant $c$ is sufficiently small, as $(t-1)^{4/5}t^{1/5} \ge t-1$.
This finishes the proof of Theorem~\ref{theorem:theorem1}.

%========================================================================================================
\section{Proof of Theorem~\ref{theorem:34holes}}\label{section_theorem34holes}

In this section we improve the lower bounds on the minimum number of 3-holes and 4-holes.
 To this end we use the notion of generated holes as introduced by Garc\'{i}a~\cite{g-nnet-12}.

Given a 5-hole $H$ in a point set $P$, 
a 3-hole in $P$ is \emph{generated by $H$} 
if it is spanned by the leftmost point $p$ of $H$ and 
the two vertices of $H$ that are not adjacent  to $p$ on the boundary of~$\conv(H)$.
Similarly, a 4-hole in $P$ is \emph{generated by $H$} 
if it is spanned by the vertices of $H$ 
with the exception of one of the points adjacent to the leftmost point of $H$ on the boundary of~$\conv(H)$.
We call a 3-hole or a 4-hole in $P$ \emph{generated} if it is generated by some 5-hole in~$P$.
We denote the number of generated 3-holes and generated 4-holes in $P$
by $h_{3|5}(P)$ and $h_{4|5}(P)$, respectively.
We also denote by $h_{3|5}(n)$ and $h_{4|5}(n)$ 
the minimum of $h_{3|5}(P)$ and $h_{4|5}(P)$, respectively, among all sets $P$ of $n$ points.

For an integer $k \ge 3$ and a point set $P$, let $h_k(P)$ be the number of $k$-holes in $P$.
We say that a point from $P$ is \emph{extremal} in~$P$ 
if it is a vertex of the polygon~$\conv(P)$.
A point from $P$ that is not extremal is \emph{inner} in~$P$.
Garc\'{i}a~\cite{g-nnet-12} proved the following relationships 
between $h_3(P)$ and $h_{3|5}(P)$
and between $h_4(P)$ and $h_{4|5}(P)$.

\begin{thm}[\cite{g-nnet-12}]
%\begin{thm}[\cite{g-nnet-12,afhhpv-lbnsc-14}]
\label{theorem:generated34holes_garcia}
Let $P$ be a set of $n$ points and let
$\gamma(P)$ be the number of extremal points of~$P$. 
Then the following two equalities are satisfied:
\begin{enumerate}[(i)] 
\item 
$h_3(P) = n^2-5n+\gamma(P)+4+h_{3|5}(P)$ and
\item 
$h_4(P) = \frac{n^2}{2}-\frac{7n}{2}+\gamma(P)+3+h_{4|5}(P)$.
\end{enumerate}
\end{thm}

The proofs of both parts of Theorem~\ref{theorem:34holes} are carried out by induction on $n$ 
similarly to the proof of Theorem~\ref{theorem:theorem1}.
The base cases follow from the fact that each set $P$ of $n\ge 10$ points contains at least one 5-hole in $P$
and thus a generated 3-hole in $P$ and a generated 4-hole in~$P$.
For the inductive step,
let $P = A \cup B$ be an $\ell$-divided set of $n$ points with 
$|A|,|B| \ge \left\lfloor \frac{n}{2} \right\rfloor$, 
where $n$ is a sufficiently large positive integer.

To show part~(\ref{thm:h_34:item1}), it suffices to prove $h_{3|5}(P) \ge \Omega(n \log^{2/3} n)$ 
as the statement then follows from Theorem~\ref{theorem:generated34holes_garcia}.
We use the recursive approach from the proof of Theorem~\ref{theorem:theorem1}, 
where we choose $r=\lfloor\log_2^{1/3} n\rfloor$.
In each step of the recursion
we either obtain
$\left\lfloor\frac{n}{4r}\right\rfloor$ pairwise disjoint $r$-holes in~$P$ or
$\left\lfloor\frac{n}{4r}\right\rfloor$ pairwise disjoint $\ell$-divided 5-holes in~$P$.

In the first case, each $r$-hole in~$P$
admits $\binom{r}{3}$ 3-holes in~$P$ and, 
by Theorem~\ref{theorem:generated34holes_garcia}, it contains
$
\binom{r}{3}-r^2+5r-r-4
$
generated 3-holes in~$P$.
Thus, in total, we count at least 
$
\frac{n}{4r}\binom{r}{3}-O(nr) \ge \Omega(n \log^{2/3} n)
$ 
generated 3-holes in~$P$.

In the second case, 
we have at least $\left\lfloor\frac{n}{4r}\right\rfloor$ $\ell$-divided 5-holes in~$P$.
Without loss of generality, 
we can assume that at least $\frac{1}{2}\left\lfloor\frac{n}{4r}\right\rfloor \ge \left\lfloor\frac{n}{8r}\right\rfloor$ 
of those $\ell$-divided 5-holes in~$P$ 
contain at least two points to the right of $\ell$,
as we otherwise continue with the horizontal reflection of~$P$, 
which has $\ell$ as the axis of reflection.
%Note that if a 5-hole is reflected, then it generates another 3-hole, however, the number of generated 3-holes both in $A$ and in $B$ remains the same after the reflection by Theorem~\ref{theorem:generated34holes_garcia}. Hence we can apply this transformation. 
Therefore we have at least $\left\lfloor\frac{n}{8r}\right\rfloor$ $\ell$-divided generated 3-holes in~$P$
and, analogously as in the proof of Theorem~\ref{theorem:theorem1}, 
we obtain 
$$
h_{3|5}(P) \ge 
2h_{3|5} \left(  \left\lfloor \frac{n}{2}\right\rfloor \right)
+\left\lfloor\frac{n}{4r}\right\rfloor \ge \Omega(n \log^{2/3} n).
$$
This finishes the proof of part~(\ref{thm:h_34:item1}).

The proof of part~(\ref{thm:h_34:item2}) is almost identical.
We choose $r=\lfloor\log_2^{1/4} n\rfloor$
and use the facts 
that every $r$-hole in~$P$  contains
$
\binom{r}{4}-\frac{r^2}{2}+\frac{7r}{2}-r-3
$
generated 4-holes in~$P$
and that every $\ell$-divided 5-hole in~$P$ generates two 4-holes in~$P$,
at least one of which is $\ell$-divided.
This finishes the proof of Theorem~\ref{theorem:34holes}.

%========================================================================================================
\section{Preliminaries for the proof of Theorem~\ref{theorem:theorem2}}\label{section_preliminaries}

Before proceeding with the proof of Theorem~\ref{theorem:theorem2}, we first introduce some notation and definitions, and state some immediate observations.

Let $a,b,c$ be three distinct points in the plane.
We denote the line segment spanned by $a$ and $b$ as  $ab$,
the ray starting at $a$ and going through $b$ as $\overrightarrow{ab}$,
and the line through $a$ and $b$ directed from $a$ to $b$ as $\dline{ab}$.
We say $c$ is to the \emph{left} (\emph{right}) of $\dline{ab}$ 
if the triple $(a,b,c)$ traced in this order 
is oriented counterclockwise (clockwise).
Note that $c$ is to the left of $\dline{ab}$ if and only if $c$ is to the right of $\dline{ba}$,
and that the triples $(a,b,c)$, $(b,c,a)$, and $(c,a,b)$ have the same orientation.
We say a point set $S$ is to the \emph{left} (\emph{right}) of $\dline{ab}$
if every point of $S$ is to the left (right) of $\dline{ab}$.

\paragraph*{Sectors of polygons}

For an integer $k \geq 3$, let $\mathcal{P}$ be a convex polygon 
with vertices $p_1,p_2,\ldots,p_k$ 
traced counterclockwise in this order.
We denote by $S(p_1,p_2,\ldots,p_k)$ the open convex region to the left of
each of the three lines $\dline{p_1p_2}$, $\dline{p_1p_k}$, and $\dline{p_{k-1}p_{k}}$.
We call the region $S(p_1,p_2,\ldots,p_k)$ a \emph{sector} of $\mathcal{P}$.
Note that every convex $k$-gon defines exactly $k$ sectors.
Figure~\ref{fig:sectors_wedges}(\subref{fig:sectors_wedges_1}) gives an illustration.

\begin{figure}[htb]
	\hfill
	\begin{subfigure}[t]{.35\textwidth}\centering
		\includegraphics[page=1,width=\textwidth]{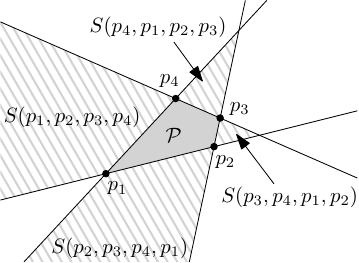}
		\caption{}
		\label{fig:sectors_wedges_1}
	\end{subfigure}
	\hfill
	\begin{subfigure}[t]{.3\textwidth}\centering
		\includegraphics[page=2,width=\textwidth]{sectors_wedges}
		\caption{}
		\label{fig:sectors_wedges_2}
	\end{subfigure}
	\hfill
	\begin{subfigure}[t]{.3\textwidth}\centering
		\includegraphics[page=3,width=\textwidth]{sectors_wedges}
		\caption{}
		\label{fig:sectors_wedges_3}
	\end{subfigure}
	\hfill
	\caption{
	(\subref{fig:sectors_wedges_1})~An example of sectors.
	(\subref{fig:sectors_wedges_2})~An example of \mbox{$a^*$-wedges} with $t=|A|-1$.
	(\subref{fig:sectors_wedges_3})~An example of \mbox{$a^*$-wedges} with $t<|A|-1$.
	}
	\label{fig:sectors_wedges}
\end{figure}

We use $\triangle(p_1,p_2,p_3)$ to denote the closed triangle with vertices $p_1,p_2,p_3$.
We also use $\square(p_1,p_2,p_3,p_4)$ to denote the closed quadrilateral with vertices $p_1,p_2,p_3,p_4$ traced in the counterclockwise order along the boundary. 

The following simple observation summarizes some properties of sectors of polygons.
\begin{obs}\label{observation:observation1}
Let $P = A \cup B$  be an $\ell$-divided 
set with no $\ell$-divided 5-hole in~$P$.
Then the following conditions are satisfied.
\begin{enumerate}[(i)]
\item \label{observation:observation1_item1} 
	Every sector of an $\ell$-divided 4-hole in $P$ is empty of points of~$P$.
\item \label{observation:observation1_item2} 
	If $S$ is a sector of a 4-hole in~$A$
	and $S$ is empty of points of~$A$, 
	then $S$ is empty of points of~$B$.
\end{enumerate}
\end{obs}

\paragraph*{\boldmath{$\ell$}-critical sets and islands}

An $\ell$-divided set $C = A \cup B$ 
is called \emph{$\ell$-critical} if
it fulfills the following two conditions.
\begin{enumerate}[(i)]
	\item 
	Neither $A$ nor $B$ is in convex position.
	\item 
	For every extremal point $x$ of~$C$, 
	one of the sets $(C\setminus \{x\}) \cap A$ and 
	$(C\setminus \{x\}) \cap B$ is in convex position.
\end{enumerate}
	
Note that every $\ell$-critical set $C = A \cup B$ contains 
at least four points in each of $A$ and~$B$.
Figure~\ref{fig:minimal_sets} shows some examples of $\ell$-critical sets.
If $P = A \cup B$ is an $\ell$-divided set with neither $A$ nor $B$ in convex position,
then there exists an $\ell$-critical island of~$P$.
This can be seen by iteratively removing extremal points so that none of the parts is in convex position after the removal.

\begin{figure}[htb]
	\hfill
	\begin{subfigure}[t]{.24\textwidth}\centering
		\includegraphics[page=1,width=\textwidth]{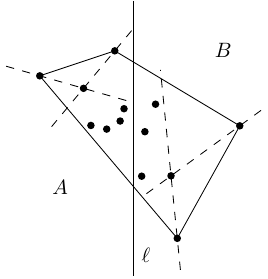}
		\caption{}
		\label{fig:minimal_sets_A}
	\end{subfigure}
	\hfill
	\begin{subfigure}[t]{.24\textwidth}\centering
		\includegraphics[page=2,width=\textwidth]{minimal_sets}
		\caption{}
		\label{fig:minimal_sets_B}
	\end{subfigure}
	\hfill
	\begin{subfigure}[t]{.24\textwidth}\centering
		\includegraphics[page=3,width=\textwidth]{minimal_sets}
		\caption{}
		\label{fig:minimal_sets_C}
	\end{subfigure}
	\hfill
	\begin{subfigure}[t]{.24\textwidth}\centering
		\includegraphics[page=4,width=\textwidth]{minimal_sets}
		\caption{}
		\label{fig:minimal_sets_D}
	\end{subfigure}
	\hfill
	\caption{Examples of $\ell$-critical sets.}
	\label{fig:minimal_sets}
\end{figure}

\paragraph*{\boldmath{$a$}-wedges and \boldmath{$\astar$}-wedges}

Let $P = A \cup B$ be an $\ell$-divided set.
For a point $a$ in~$A$, 
the rays $\overrightarrow{a a'}$ for all $a' \in A \setminus \{a\}$ 
partition the plane 
into $|A|-1$ regions.
We call the closures of those regions \emph{$a$-wedges}
and 
label them as $W^{(a)}_1,\ldots,W^{(a)}_{|A|-1}$ in the clockwise order around~$a$,
where $W^{(a)}_1$ is the topmost $a$-wedge that intersects~$\ell$.
Let $t^{(a)}$ be the number of $a$-wedges that intersect~$\ell$.
Note that 
$W^{(a)}_1,\ldots,W^{(a)}_{t^{(a)}}$ are the $a$-wedges that intersect $\ell$
sorted in top-to-bottom order on~$\ell$.
Also note that all $a$-wedges are convex if $a$ is an inner point of~$A$, 
and that there exists exactly one non-convex $a$-wedge otherwise.
The indices of the $a$-wedges are considered modulo $|A|-1$.
In particular, $W_{0}^{(a)} = W_{|A|-1}^{(a)}$
and $W_{|A|}^{(a)} = W_{1}^{(a)}$.

If $A$ is not in convex position, 
we denote the rightmost inner point of $A$ as $\astar$
and write $t\mathrel{\mathop:}=t^{(\astar)}$ and 
$W_k \mathrel{\mathop:}= W^{(\astar)}_k$ for $k=1,\ldots,|A|-1$.
Recall that $\astar$ is unique, 
since all points have distinct $x$-coordinates.
Figures~\ref{fig:sectors_wedges}(\subref{fig:sectors_wedges_2}) and \ref{fig:sectors_wedges}(\subref{fig:sectors_wedges_3}) give an illustration.

We set
$w_k \mathrel{\mathop:}= |B \cap W_k|$
and label the points of $A$ 
so that $W_k$ is bounded by the rays $\overrightarrow{\astar a_{k-1}}$ 
and $\overrightarrow{\astar a_{k}}$ for $k=1,\ldots,|A|-1$.
Again, the indices are considered modulo $|A|-1$.
In particular, $a_0 = a_{|A|-1}$ and $a_{|A|}=a_1$.

\begin{obs}\label{observation:observation3}
Let $P = A \cup B$ be an $\ell$-divided set with $A$ not in convex position.
Then the points $a_1,\ldots,a_{t-1}$ lie to the right of~$\astar$
and 
the points $a_t,\ldots,a_{|A|-1}$ lie to the left of~$\astar$.
\end{obs}

%========================================================================================================
\section{Proof of Theorem~\ref{theorem:theorem2}}\label{section_convex}
\label{section_theorem2}

First, we give a high-level overview of the main ideas of the proof of Theorem~\ref{theorem:theorem2}.
We proceed by contradiction and we suppose that there is no $\ell$-divided 5-hole in a given $\ell$-divided set $P=A \cup B$ with $|A|,|B| \ge 5$ and with neither $A$ nor $B$ in convex position.
If $|A|, |B| = 5$, 
then the statement follows from the result of Harborth~\cite{Ha78}.
Thus we assume that $|A| \ge 6$ or $|B| \ge 6$.
We reduce $P$ to an island $Q$ of $P$
by iteratively removing points from the convex hull
until one of the two parts $Q \cap A$ and $Q \cap B$ contains exactly five points
or $Q$ is $\ell$-critical with $|Q \cap A|,|Q \cap B| \ge 6$.
If $|Q \cap A| = 5$ and $|Q \cap B| \ge 6$ or vice versa, then we reduce $Q$ to an island of~$Q$ with eleven points and, using a computer-aided result (Lemma~\ref{lemma:lemma9}), we show that there is an $\ell$-divided 5-hole in that island and hence in~$P$.
If $Q$ is $\ell$-critical with $|Q \cap A|,|Q \cap B| \ge 6$, 
then we show that $|A \cap \partial \conv(Q)|,|B \cap \partial \conv(Q)| \le 2$ and that, if $|A \cap \partial \conv(Q)|=2$, 
then $a^*$ is the only inner point of $Q \cap A$ and similarly for $B$ (Lemma~\ref{lemma:lemma4}).
Without loss of generality, we assume that $|A \cap \partial \conv(Q)|=2$ and thus $a^*$ is the only inner point of $Q \cap A$.
Using this assumption, we prove that $|Q \cap B| < |Q \cap A|$ (Proposition~\ref{proposition:proposition15}).
By exchanging the roles of $Q \cap A$ and $Q \cap B$, we obtain $|Q\cap A| \le |Q \cap B|$ (Proposition~\ref{proposition:proposition16}), which gives a contradiction. 

To prove that $|Q \cap B| < |Q \cap A|$, 
we use three results about the sizes of the parameters $w_1,\dots,w_t$ for the $\ell$-divided set~$Q$,
that is, about the numbers of points of $Q \cap B$ in the \mbox{$a^*$-wedges} $W_1,\dots,W_t$ of~$Q$.
We show that if we have $w_i=2=w_j$ for some $1 \leq i < j \leq t$, then $w_k=0$ for some $k$ with $i<k<j$ (Lemma~\ref{lemma:lemma7}).
Further, for any three or four consecutive \mbox{$a^*$-wedges} whose union is convex and contains at least four points of $Q \cap B$, each of those \mbox{$a^*$-wedges} contains at most two such points (Lemma~\ref{lemma:lemma13}). %each wedge $W_i$ of them satisfies $w_i\leq 2$.
Finally, we show that $w_1,\dots,w_t \leq 3$ (Lemma~\ref{lemma:lemma14}).
The proofs of Lemmas~\ref{lemma:lemma13} and~\ref{lemma:lemma14} rely on some results about small $\ell$-divided sets with computer-aided proofs (Lemmas~\ref{lemma:lemma10}, \ref{lemma:lemma11}, and \ref{lemma:lemma12}).
Altogether, this is sufficient to show that $|Q \cap B| < |Q \cap A|$.

We now start the proof of Theorem~\ref{theorem:theorem2} by showing that
if there is an $\ell$-divided 5-hole in the intersection of $P$ with a union of consecutive \mbox{$a^*$-wedges}, then there is an $\ell$-divided 5-hole in~$P$.

\begin{lem}\label{lemma:lemma5}
Let $P = A \cup B$ be an $\ell$-divided set 
with $A$ not in convex position.
For integers $i,j$ with $1\le i \le j \le t$,
let $W \mathrel{\mathop:}= \bigcup_{k=i}^j W_k$ and $Q \mathrel{\mathop:}= P \cap W$.
If there is an $\ell$-divided 5-hole in~$Q$,
then there is an $\ell$-divided 5-hole in~$P$.
\end{lem}

\begin{proof}
If $W$ is convex then $Q$ is an island of~$P$ and the statement immediately follows.
Hence we assume that $W$ is not convex.
The region $W$ is bounded by the rays $\overrightarrow{\astar a_{i-1}}$ and $\overrightarrow{\astar a_j}$
and all points of $P \setminus Q$ lie in the 
convex region $\mathbb{R}^2 \setminus W$;
see Figure~\ref{fig:dividedHolesInWedgeUnions}.

\begin{figure}[htb]
	\hfill
	\begin{subfigure}[t]{.3\textwidth}\centering
		\includegraphics[page=1,width=\textwidth]{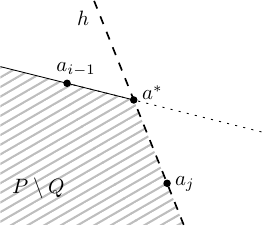}
		\caption{}
		\label{fig:dividedHolesInWedgeUnions1}
	\end{subfigure}
	\hfill
	\begin{subfigure}[t]{.3\textwidth}\centering
		\includegraphics[page=2,width=\textwidth]{partitionedHolesInWedgeUnions}
		\caption{}
		\label{fig:dividedHolesInWedgeUnions2}
	\end{subfigure}
	\hfill
	\begin{subfigure}[t]{.3\textwidth}\centering
		\includegraphics[page=3,width=\textwidth]{partitionedHolesInWedgeUnions}
		\caption{}
		\label{fig:dividedHolesInWedgeUnions3}
	\end{subfigure}
	\hfill
	\caption{Illustration of the proof of Lemma~\ref{lemma:lemma5}. 
	(\subref{fig:dividedHolesInWedgeUnions1})~The point $a_j$ is to the right of $a^*$. 
	(\subref{fig:dividedHolesInWedgeUnions2})~The point $a_j$ is to the left of $a^*$. 
	(\subref{fig:dividedHolesInWedgeUnions3})~The hole $H$ properly intersects the ray $\protect\overrightarrow{a^*a_j}$. 
	The boundary of the convex hull of $H$ is drawn red and the convex hull of $H'$ is drawn blue.}
	\label{fig:dividedHolesInWedgeUnions}
\end{figure}

Since $W$ is non-convex and 
every $a^*$-wedge contained in $W$ intersects~$\ell$, 
at least one of the points $a_{i-1}$ and $a_j$ lies to the left of $a^*$.
Moreover, the points $a_i,\ldots,a_{j-1}$ are to the right of $a^*$ by Observation~\ref{observation:observation3}.
Without loss of generality, we assume that $a_{i-1}$ is to the left of $a^*$,
as otherwise we consider the vertical reflection of the whole point set~$P$.

If $a_j$ is to the left of $a^*$,
then we let $h$ be the closed halfplane determined by the vertical line through $a^*$
such that $a_{i-1}$ and $a_j$ lie in $h$.
Otherwise, if $a_j$ is to the right of $a^*$,
then we let $h$ be the closed halfplane determined by the line $\overline{a^*a_{j}}$
such that $a_{i-1}$ lies in $h$.
In either case, $h \cap A \cap Q = \{ a^*, a_{i-1}, a_j\}$.

Let $H$ be an $\ell$-divided 5-hole in~$Q$.
We say that $H$ \emph{properly intersects} a ray~$r$
if the interior of $\conv(H)$ intersects~$r$.
Now we show that if $H$ properly intersects the ray $\overrightarrow{a^* a_j}$,
then $H$ contains $a_{i-1}$.
Assume there are points $p,q\in H$ such that 
the relative interior of $pq$ intersects $r \mathrel{\mathop:}= \overrightarrow{a^* a_j}$.
%\birgit{segments don't have an interior ... the relative interior of $pq$?} 
Since $r$ lies in $h$ and neither of $p$ and $q$ lies in~$r$, 
at least one of the points $p$ and $q$ lies in $h \setminus r$.
Without loss of generality, we assume $p \in h \setminus r$.
From $h \cap A \cap Q = \{ a^*, a_{i-1}, a_j\}$ we have $p = a_{i-1}$.
By symmetry, 
if $H$ properly intersects the ray $\overrightarrow{a^* a_{i-1}}$,
then $H$ contains $a_j$.

Suppose for contradiction 
that $H$ properly intersects both rays 
$\overrightarrow{\astar a_{i-1}}$ and $ \overrightarrow{\astar a_j}$.
Then $H$ contains the points $a_{i-1},a_j,x,y,z$ for some points $x,y,z \in Q$, 
where $a_{i-1} x$ intersects $\overrightarrow{a^* a_j}$, and 
$a_j z$ intersects $\overrightarrow{a^* a_{i-1}}$.
Observe that $z$ is to the left of $\dline{a_{i-1} \astar}$ and
that $x$ is to the right of $\dline{a_j \astar}$.
If $a_j$ lies to the right of~$a^*$,
then $z$ is to the left of $a^*$, and thus $z$ is in~$A$; see Figure~\ref{fig:dividedHolesInWedgeUnions}(\subref{fig:dividedHolesInWedgeUnions1}).
However, this is impossible as $z$ also lies in~$h$.
Hence, $a_j$ lies to the left of~$a^*$; 
see Figure~\ref{fig:dividedHolesInWedgeUnions}(\subref{fig:dividedHolesInWedgeUnions2}).
As $x$ and $z$ are both to the right of $\astar$, 
the point $\astar$ is inside the convex quadrilateral $\square(a_{i-1},a_j,x,z)$.
This contradicts the assumption that $H$ is a 5-hole in~$Q$.

So assume that $H$ properly intersects exactly one of the rays
$\overrightarrow{\astar a_{i-1}}$ and $\overrightarrow{\astar a_j}$, 
say $\overrightarrow{\astar a_{j}}$; 
see Figure~\ref{fig:dividedHolesInWedgeUnions}(\subref{fig:dividedHolesInWedgeUnions3}).
In this case, $H$ contains~$a_{i-1}$.
The interior of the triangle $\triangle(\astar,a_{i-1},a_j)$ is empty of points of~$Q$,
since the triangle is contained in~$h$.
Moreover, $\conv(H)$ cannot intersect the line that determines $h$ both strictly above and strictly below~$a^*$.
Thus, all remaining points of $H \setminus \{ a_{i-1} \}$ 
lie to the right of $\dline{ a_{i-1} \astar}$ and to the right of $\dline{ a_j \astar}$.
If $H$ is empty of points of $P \setminus Q$, we are done.
Otherwise, we let $H' \mathrel{\mathop:}= (H  \setminus \{a_{i-1}\})\cup \{p'\}$ 
where $p' \in P \setminus Q$ is a point 
inside $\triangle(\astar,a_{i-1},a_j)$ closest to $\overline {a_j \astar}$.
Note that the point $p'$ might not be unique.
By construction, $H'$ is an $\ell$-divided 5-hole in~$P$.
An analogous argument shows that there is an $\ell$-divided 5-hole in $P$ if $H$ properly intersects $\overrightarrow{\astar a_{i-1}}$.

Finally, if $H$ does not properly intersect any of the rays 
$\overrightarrow{\astar a_{i-1}}$ and $ \overrightarrow{\astar a_j}$, 
then $\conv(H)$ contains no point of $P \setminus Q$ in its interior,
and hence $H$ is an $\ell$-divided 5-hole in~$P$.
\end{proof}

%\subsection{Sequences of \mbox{$a^*$-wedges} with at most two points of~$B$}
\subsection{Sequences of \texorpdfstring{\boldmath{$a^*$}}{a*}-wedges with at most two points of~\texorpdfstring{\boldmath{$B$}}{B}}
\label{section_212}

In this subsection we consider an $\ell$-divided set $P = A \cup B$ with $A$ not in convex position.
We consider the union $W$ of consecutive \mbox{$\astar$-wedges},
each containing at most two points of~$B$, and 
derive an upper bound on the number of points of $B$ that lie in $W$
if there is no $\ell$-divided 5-hole in $P \cap W$; see Corollary~\ref{corollary:corollary8}.

\begin{obs}\label{observation:observation4}
Let $P = A \cup B$ be an $\ell$-divided set with $A$ not in convex position. 
Let $W_k$ be an $\astar$-wedge with $w_k \ge 1$ and $1 \le k \le t$
and let $b$ be the leftmost point in $W_k \cap B$.
Then the points $\astar$, $a_{k-1}$, $b$, and $a_k$ form an $\ell$-divided 4-hole in~$P$.
\end{obs}

From Observation~\ref{observation:observation1}(\ref{observation:observation1_item1})
and Observation~\ref{observation:observation4} we obtain the following result.

\begin{obs}\label{observation:observation5}
Let $P = A \cup B$ be an $\ell$-divided set with $A$ not in convex position 
and with no $\ell$-divided 5-hole in~$P$. 
Let $W_k$ be an $\astar$-wedge with $w_k \ge 2$ and $1 \le k \le t$
and let $b$ be the leftmost point in $W_k \cap B$.
For every point $b'$ in $(W_k \cap B) \setminus \{b\}$,
the line $\overline{bb'}$ intersects the segment $a_{k-1}a_k$.
Consequently, $b$ is inside $\triangle(a_{k-1},a_k,b')$,
to the left of $\dline{a_kb'}$,
and to the right of $\dline{a_{k-1}b'}$.
\end{obs}

The following lemma states that there is an $\ell$-divided 5-hole in $P$ 
if two consecutive \mbox{$\astar$-wedges} both contain exactly two points of~$B$.

\begin{lem}\label{lemma:lemma6}
Let $P = A \cup B$ be an $\ell$-divided set 
with $A$ not in convex position and with $|A|,|B| \ge 5$.
Let $W_i$ and $W_{i+1}$ be consecutive \mbox{$\astar$-wedges} with $w_i=2=w_{i+1}$ and $1 \leq i < t$.
Then there is an $\ell$-divided 5-hole in~$P$. 
\end{lem}

\begin{proof}
The overall idea of the proof is as follows.
We suppose for contradiction that there is no $\ell$-divided 5-hole in~$P$.
Then we prove a sequence of structural facts on the layout of the points of~$P$
forced by this assumption.
Eventually we show that the structure of the point set $P$ resembles 
the point set from Figure~\ref{fig:problem_2_2_page35}(\subref{fig:problem_2_2_page5}).
In particular, we arrive at the conclusion that
$|B| = 4$, which contradicts our assumption $|B| \ge 5$.

Suppose for contradiction that there is no $\ell$-divided 5-hole in~$P$.
Let $W \mathrel{\mathop:}= W_i \cup W_{i+1}$ and let $Q \mathrel{\mathop:}= P \cap W$.
By Lemma~\ref{lemma:lemma5}, 
there is also no $\ell$-divided 5-hole in~$Q$.
We label the points in $B \cap W_i$ as $b_{i-1}$ and $b_i$ 
so that $b_{i-1}$ is to the right of $b_i$.
Similarly, we label the points in $B\cap W_{i+1}$ as $b_{i+1}$ and $b_{i+2}$ 
so that $b_{i+2}$ is to the right of $b_{i+1}$.
By Observation~\ref{observation:observation5}, 
the point $a_i$ is to the right of $\dline{b_ib_{i-1}}$ and 
to the left of $\dline{b_{i+1}b_{i+2}}$.
If the points $b_{i-1},b_i,b_{i+1},b_{i+2}$ are in convex position, 
then $a_i,b_{i+1},b_{i+2},b_{i-1},b_i$ form an $\ell$-divided 5-hole in~$P$;
see Figure~\ref{fig:problem_2_2}(\subref{fig:problem_2_2_page1}).
Thus, we assume the points $b_{i-1},b_i,b_{i+1},b_{i+2}$ are not in convex position.
Without loss of generality, we assume that the line $\overline{b_ib_{i-1}}$ intersects the segment $b_{i+1}b_{i+2}$,
as otherwise we consider the vertical reflection of the whole point set~$P$.

\begin{claim}
\label{claim_lemma6_claim_a}
The segments $a_ib_{i-1}$ and $b_ib_{i+1}$ intersect.
\end{claim}
As $\overline{b_ib_{i-1}}$ intersects $a_ia_{i-1}$ and $b_{i+1}b_{i+2}$, 
the point $b_{i-1}$ lies in the triangle $\triangle(b_i,b_{i+1},b_{i+2})$. 
Moreover, $b_{i-1}$ is to the right of $\dline{b_{i+1}b_i}$, 
$a_i$ is to the left of $\dline{b_{i+1}b_i}$,
$b_i$ is to the left of $\dline{a_ib_{i-1}}$, and 
$b_{i+1}$ is to the right of $\dline{a_ib_{i-1}}$.
Consequently, the points $a_i,b_{i+1},b_{i-1},b_i$ form an $\ell$-divided 4-hole in~$P$,
and, in particular, the segments $a_ib_{i-1}$ and $b_ib_{i+1}$ intersect; see Figure~\ref{fig:problem_2_2}(\subref{fig:problem_2_2_page2}).
This finishes the proof of Claim~\ref{claim_lemma6_claim_a}.

	\begin{figure}[htb]
		\hfill
		\begin{subfigure}[t]{.4\textwidth}\centering
			\includegraphics[page=1]{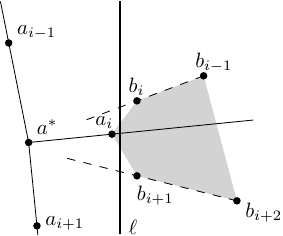}
			\caption{}
			\label{fig:problem_2_2_page1}
		\end{subfigure}
		\hfill
		\begin{subfigure}[t]{.5\textwidth}\centering
			\includegraphics[page=2]{problem_2_2}
			\caption{}
			\label{fig:problem_2_2_page2}
		\end{subfigure}
		\hfill
		\caption{
		(\subref{fig:problem_2_2_page1})~If $b_{i-1},b_i,b_{i+1},b_{i+2}$ are in convex position, then there is an $\ell$-divided 5-hole in~$P$. 
		(\subref{fig:problem_2_2_page2})~The points $\astar,a_{i+1},a_i,a_{i-1}$ form a 4-hole in~$P$.}
		\label{fig:problem_2_2}
	\end{figure}
	
The points $a_{i-1},b_i,b_{i-1},b_{i+2}$ are in convex position because 
$a_{i-1}$ is the leftmost and 
$b_{i+2}$ is the rightmost of those four points and because
both $a_{i-1}$ and $b_{i+2}$ lie to the left of $\dline{b_ib_{i-1}}$.
Moreover, the points $a_{i-1},b_i,b_{i-1},b_{i+2}$ form an $\ell$-divided 4-hole in $P$ 
as $\square(a_{i-1},b_i,b_{i-1},b_{i+2})$ lies in $W$ and $w_i=w_{i+1}=2$.
	
\begin{claim}
\label{claim_lemma6_claim_b}
Among the four points $b_{i+2},b_{i-1},b_{i+1},a_{i+1}$, the clockwise order around $b_{i+2}$ is $a_{i+1},b_{i+1},b_{i-1}$.
\end{claim}
The point $b_{i+2}$ is the rightmost of those four points.
By Observation~\ref{observation:observation5}, 
$b_{i+1}$ lies to the right of $\dline{a_ib_{i+2}}$
and $a_{i+1}$ lies to the right of $\dline{b_{i+1}b_{i+2}}$.
Since $b_{i-1} \in W_i$ and $b_{i+2} \in W_{i+1}$, 
the point $b_{i-1}$ lies to the left of $\dline{a_ib_{i+2}}$.
This finishes the proof of Claim~\ref{claim_lemma6_claim_b}.

\begin{claim}
\label{claim_lemma6_claim_c}
The points $b_{i+2},b_{i-1},b_{i+1},a_{i+1}$ are not in convex position.
\end{claim}

Suppose for contradiction that 
the points $b_{i+2},b_{i-1},b_{i+1},a_{i+1}$ form a convex quadrilateral.
Due to the clockwise order around $b_{i+2}$,
the convex quadrilateral is $\square(b_{i+2},b_{i-1},b_{i+1},\allowbreak a_{i+1})$.
The only points of $P$ that can lie in the interior of this quadrilateral
are $a^*$, $a_{i-1}$, $a_i$, and $b_i$.
Since the triangle $\triangle(b_{i+2},b_{i+1},a_{i+1})$ is contained in $W_{i+1}$,
it contains neither of the points $a^*$, $a_{i-1}$, $a_i$, and $b_i$.
Since the triangle $\triangle(b_{i+2},b_{i-1},b_{i+1})$ is contained in the convex hull of~$B$,
it does not contain $a^*$, $a_{i-1}$, nor $a_i$.
Moreover, as $b_{i-1}$ lies in the triangle $\triangle(b_{i},b_{i+1},b_{i+2})$, 
the triangle $\triangle(b_{i+2},b_{i-1},b_{i+1})$ also does not contain $b_i$.
Thus the quadrilateral $\square(b_{i+2},b_{i-1},b_{i+1},a_{i+1})$ is empty of points of~$P$.
By Observation~\ref{observation:observation1}(\ref{observation:observation1_item1}), 
the two sectors $S(a_{i-1},b_i,b_{i-1},b_{i+2})$ and $S(b_{i+2},b_{i-1},b_{i+1},a_{i+1})$ 
contain no point of~$P$.
Since every point of $B \setminus \{b_{i-1},b_i,b_{i+1},b_{i+2}\}$ is  either in $S(a_{i-1},b_i,b_{i-1},b_{i+2})$ or in $S(b_{i+2},b_{i-1},b_{i+1},a_{i+1})$, 
we have $B = \{b_{i-1},b_i,b_{i+1},b_{i+2}\}$. 
This contradicts the assumption that $|B| \ge 5$ and finishes the proof of Claim~\ref{claim_lemma6_claim_c}.
	
In particular, the point $b_{i+1}$ lies in the triangle  $\triangle(b_{i-1},a_{i+1},b_{i+2})$,
since $a_{i+1}$ is the leftmost  
and $b_{i+2}$ is the rightmost of the points $b_{i+2},b_{i-1},b_{i+1},a_{i+1}$ 
and since $b_{i-1}$ lies in $W_i$. 
The red area in Figure~\ref{fig:problem_2_2}(\subref{fig:problem_2_2_page2}) gives an illustration.

Consequently, the point $a_{i+1}$ lies to the left of $\dline{b_{i+1}b_{i-1}}$.
By Observation~\ref{observation:observation1}(\ref{observation:observation1_item1}), 
the point $a_{i+1}$ is not in the sector $S(b_{i+1},b_{i-1},b_i,a_i)$,
as otherwise the points $b_{i+1},b_{i-1},b_i,a_i,a_{i+1}$ 
form an $\ell$-divided 5-hole in~$P$.
Thus the point $a_{i+1}$ lies to the left of $\dline{a_ib_i}$; see Figure~\ref{fig:problem_2_2}(\subref{fig:problem_2_2_page2}).

\begin{claim}
\label{claim_lemma6_claim_d}
The points $\astar,a_{i+1},a_i,a_{i-1}$ are not in convex position.
\end{claim}
The points $\astar,a_{i+1},a_i,a_{i-1}$ do not form a 4-hole in~$P$ because
otherwise $b_i$ lies in the sector $S(a_{i-1},\astar,a_{i+1},a_i)$ 
and forms a $5$-hole together with $a_{i-1},\astar,a_{i+1},a_i$,
which is impossible by Observation~\ref{observation:observation1}(\ref{observation:observation1_item2}).
This finishes the proof of Claim~\ref{claim_lemma6_claim_d}.

\begin{claim}
\label{claim_lemma6_claim_e}
The point $a^*$ is inside the triangle $\triangle(a_{i-1},a_{i+1},a_i)$.
\end{claim}
The point $a_i$ is not inside $\triangle(a_{i-1},a_{i+1},a^*)$,
since, by Observation~\ref{observation:observation3}, %(\ref{observation:observation3_item2})
$a_i$ is to the right of $\astar$ and 
since $a^*$ is the rightmost inner point of~$A$.
Since $a_{i-1}$ is to the left of $\dline{\astar a_i}$ 
and $a_{i+1}$ is to the right of $\dline{\astar a_i}$,
$\astar$ is the inner point of $\astar,a_{i+1},a_i,a_{i-1}$.
Figure~\ref{fig:problem_2_2_page4} gives an illustration.
This finishes the proof of Claim~\ref{claim_lemma6_claim_e}.

	\begin{figure}[htb]
		\centering
		\includegraphics[page=4]{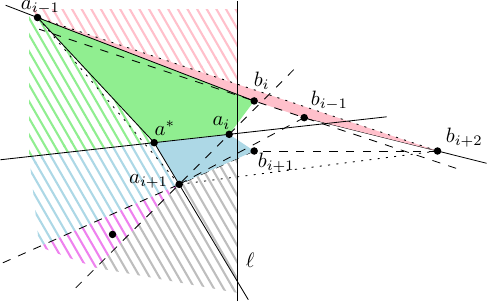}
		\caption{Location of the points of $A \setminus Q$.}
		\label{fig:problem_2_2_page4}
	\end{figure}	
	
\begin{claim}
\label{claim_lemma6_claim_f}
All points of $B\setminus Q$ lie in \mbox{$\astar$-wedges} below~$W_{i+1}$. 
\end{claim}
Since $|B| \ge 5$, there is another $\astar$-wedge besides $W_i$ and $W_{i+1}$ that intersects~$\ell$.
Now we show that all points of $B \setminus Q$ lie 
in \mbox{$a^*$-wedges} below $W_{i+1}$.
The rays $\overrightarrow{b_ia_{i-1}}$ and $\overrightarrow{b_{i-1}b_{i+2}}$ both start in $W_i$ and then leave $W_i$.
Moreover, the segment $b_ia_{i-1}$ intersects $\ell$ and
$b_{i-1}b_{i+2}$ intersects $\overrightarrow{\astar a_i}$.
As both $b_i$ and $b_{i-1}$ lie to the right of $\dline{a_{i-1}b_{i+2}}$,
all points of $B\setminus Q$ that lie in an $\astar$-wedge above $W_i$ also lie in the sector $S(a_{i-1},b_i,b_{i-1},b_{i+2})$.
We recall that, by Observation~\ref{observation:observation1}(\ref{observation:observation1_item1}), 
the sector $S(a_{i-1},b_i,b_{i-1},b_{i+2})$ is empty of points of~$P$.
This finishes the proof of Claim~\ref{claim_lemma6_claim_f}.

	\begin{claim}\label{claim:lemma6_claim1}
	We have $i=1$.
	That is, $W_i$ is the topmost $a^*$-wedge that intersects~$\ell$.
	\end{claim}
	By Observation~\ref{observation:observation3}, 
	$a_{i+1}$ lies to the right of $\astar$.
	Since $a_i$ and $a_{i+1}$ are both to the right of $a^*$
	and since $a^*$ is inside the triangle $\triangle(a_{i-1},a_{i+1},a_i)$,
	the point $a_{i-1}$ is to the left of~$a^*$.
	By Observation~\ref{observation:observation3}, 
	we have $i=1$.
	This proves Claim~\ref{claim:lemma6_claim1}.
	
	\begin{claim}\label{claim:lemma6_claim2}
	All points of $A \setminus Q$ 
	lie to the left of $\overline{a_{i+1}a_i}$,
	to the right of $\overline{a_{i+1}b_{i+1}}$,
	and to the right of $\overline{a^*a_{i+1}}$.
	\end{claim}
	The violet area in Figure~\ref{fig:problem_2_2_page4}
	gives an illustration where the remaining points of $A \setminus Q$ lie.
	We recall that the sector $S(a_{i-1},b_i,b_{i-1},b_{i+2})$
	(red shaded area in Figure~\ref{fig:problem_2_2_page4})
	is empty of points of~$P$.
	By Observation~\ref{observation:observation4},
	both sets $\{\astar,a_i,b_i,a_{i-1}\}$ 
	and $\{\astar,a_{i+1},b_{i+1},a_i\}$
	form $\ell$-divided 4-holes in~$P$.
	By Observation~\ref{observation:observation1}(\ref{observation:observation1_item1}), 
	the two sectors 
	$S(\astar,a_i,b_i,a_{i-1})$ (green shaded area in Figure~\ref{fig:problem_2_2_page4})
	and
	$S(\astar,a_{i+1},b_{i+1},a_i)$ (blue shaded area in Figure~\ref{fig:problem_2_2_page4})
	are thus empty of points of~$P$.
	Therefore, 
	no point of $A \setminus Q$ lies to the left of 
	$\dline{a_{i+1}b_{i+1}}$.
	Since $W$ is non-convex,
	every point of~$P$ that is to the left of $\dline{a^*a_{i+1}}$ lies in $Q$.
	Thus every point of $A \setminus Q$ lies to the right of $\dline{a^*a_{i+1}}$.
	Moreover, no point $a$ of $A \setminus Q$ lies
	to the right of $\dline{a_{i+1}a_i}$ 
	(gray area in Figure~\ref{fig:problem_2_2_page4})
	because otherwise, 
	$a_{i+1}$ is an inner point of $\triangle(a_i,\astar,a)$,
	which is impossible since $\astar$ is the rightmost inner point of $A$ and $a_{i+1}$ is to the right of $a^*$.
	This finishes the proof of Claim~\ref{claim:lemma6_claim2}.
	
	Now we have restricted where the points of $A \setminus Q$ lie.
	In the rest of the proof we show the following claim.
	We will then use the sectors $S(b_{i+2},b_{i+1},a_{i+1},a_{i+2})$ and 
	$S(a_{i-1},b_i,b_{i-1},b_{i+2})$ to argue that $|B| = |B \cap Q| = 4$, 
	which then contradicts the assumption $|B| \ge 5$.
	
	\begin{claim}\label{claim:lemma6_claim3}
	The points $b_{i+2},b_{i+1},a_{i+1},a_{i+2}$ form an $\ell$-divided 4-hole in~$P$.
	\end{claim}
	
	We consider $a_{i+2}$ 
	and show that the points $a_{i+1},\astar,a_{i-1},a_{i+2}$ are in convex position.
	It suffices to show that $a_{i+2}$ does not lie 
	in the triangle $ \triangle(\astar,a_{i-1},a_{i+1})$
	because of the cyclic order of $A \setminus \{a^*\}$ around~$a^*$.
	Recall that $\astar $ lies inside the triangle $\triangle(a_{i-1},a_{i+1},a_i)$,
	that $b_{i+1}$ lies inside the triangle $\triangle(a_i,a_{i+1},b_{i+2})$,
	and that $b_{i-1}$ lies inside the triangle $\triangle(a_{i-1},a_i,b_{i+2})$.
	Since the triangles $\triangle(a_{i-1},a_{i+1},a_i)$, $\triangle(a_i,a_{i+1},b_{i+2})$, and $\triangle(a_{i-1},a_i,b_{i+2})$
	are oriented counterclockwise along the boundary,
	the point $a_i$ lies inside $\triangle(a_{i-1},a_{i+1},b_{i+2})$.
	Thus also the points $\astar,b_i,b_{i+1}$ lie 
	in the triangle $\triangle(a_{i-1},a_{i+1},b_{i+2})$.
	Consequently, the triangle $\triangle(\astar,a_{i-1},a_{i+1})$
	is contained in the union of the sectors
	$S(a_{i+1},b_{i+1},a_i,\astar)$ 
	(blue shaded area in Figure~\ref{fig:problem_2_2_page4}) 
	and $S(\astar,a_i,b_i,a_{i-1})$ 
	(green shaded area in Figure~\ref{fig:problem_2_2_page4}).	
	Thus $a_{i+2}$ does not lie in the triangle $\triangle(\astar,a_{i-1},a_{i+1})$ 
	and the points
	$a_{i+1},\astar,a_{i-1},a_{i+2}$ are in convex position.
	
	We now show that the sector $S(a_{i+1},\astar,a_{i-1},a_{i+2})$ 
	is empty of points of~$P$.
	If the quadrilateral $\square(a_{i+1},\astar,a_{i-1},a_{i+2})$ 
	is not empty of points of~$P$, 
	then there is a point $a_{i-1}'$ of $A$ in $\triangle(\astar,a_{i-1},a_{i+2})$.
	This is because $\triangle(\astar,a_{i+2},a_{i+1})$ 
	is empty of points of $A$ due to the cyclic order 
	of $A \setminus \{a^*\}$ around $a^*$.
	We can choose $a_{i-1}'$ to be a point
	that is closest to the line $\overline{\astar a_{i+2}}$ 
	among the points of $A$ inside $\triangle(\astar,a_{i+2},a_{i+1})$. 
	If the quadrilateral $\square(a_{i+1},\astar,a_{i-1},a_{i+2})$ is empty of points of~$P$, 
	then we set $a_{i-1}'\mathrel{\mathop:}=a_{i-1}$.
	
	By the choice of $a_{i-1}'$, 
	the quadrilateral $\square(a_{i+1},\astar,a_{i-1}',a_{i+2})$ 
	is empty of points of~$P$. 
	Since $a_{i+1}$ and $a_{i+2}$ are consecutive in the order around $\astar$, 
	no point of $A$ lies in the sector $S(a_{i+1},\astar,a_{i-1}',a_{i+2})$.
	By Observation~\ref{observation:observation1}(\ref{observation:observation1_item2}),
	the sector $S(a_{i+1},\astar,a_{i-1}',a_{i+2})$  (gray shaded area in Figure~\ref{fig:problem_2_2_page35}(\subref{fig:problem_2_2_page5}))
	is empty of points of~$P$.
	Since 
	the sector $S(a_{i+1},\astar,a_{i-1},a_{i+2})$
	is a subset of 
	$S(a_{i+1},\astar,a_{i-1}',a_{i+2})$,
	the sector $S(a_{i+1},\astar,a_{i-1},a_{i+2})$ 
	is empty of points of~$P$.	
	
	\begin{figure}[htb]
		\hfill
		\begin{subfigure}[t]{.65\textwidth}\centering
			\includegraphics[page=5,width=\textwidth]{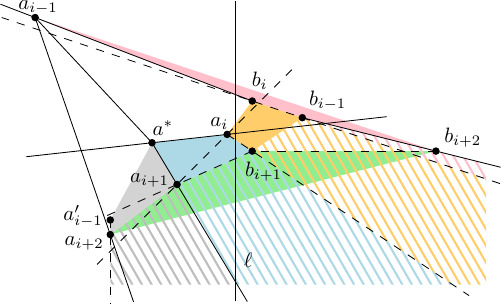}
			\caption{}
			\label{fig:problem_2_2_page5}
		\end{subfigure}
		\hfill
		\begin{subfigure}[t]{.25\textwidth}\centering
			\includegraphics[page=3,width=\textwidth]{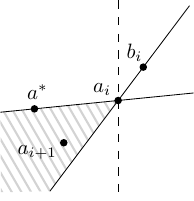}
			\caption{}
			\label{fig:problem_2_2_page3}
		\end{subfigure}
		\hfill
		\caption{
		(a)~Location of the points of $B\setminus Q$.
		(b)~The point $a_{i+1}$ lies to the left of $a_i$.}
		\label{fig:problem_2_2_page35}
	\end{figure}
	
	We show that $a_{i+1}$ is to the left of $a_i$ and to the right of $a_{i+2}$.
	Recall that $a_i$ lies to the right of $\astar$ 
	and to the left of $b_i$.
	The point $b_i$ lies to the left of $\dline{\astar a_i}$ and
	the point $a_{i+1}$ lies to the right of this line; see Figure~\ref{fig:problem_2_2_page35}(\subref{fig:problem_2_2_page3}).
	The point $a_{i+1}$ then lies to the left of $a_i$, since we know already that $a_{i+1}$ lies to the left of $\dline{a_i b_i}$.
	Recall that $a_{i+1}$ is to the right of $\astar$.
	Consequently, the point $a_{i+2}$ lies to the left of $a_{i+1}$, as $a_{i+2}$ 
	lies to the right of $\dline{\astar a_{i+1}}$ and to the left of $\dline{a_{i+1}a_i}$ by Claim~\ref{claim:lemma6_claim2}.

	Now we are ready to prove that
	the points $b_{i+2},b_{i+1},a_{i+1},a_{i+2}$ 
	form an $\ell$-divided 4-hole in~$P$ (green area in Figure~\ref{fig:problem_2_2_page35}(\subref{fig:problem_2_2_page5})).
	Recall that $b_{i+2}$ and $a_{i+2}$ both 
	lie to the right of $\dline{a_{i+1}b_{i+1}}$,
	and that $a_{i+2}$ is the leftmost and $b_{i+2}$ is the rightmost of those four points.
	Altogether,
	we see that the points $b_{i+2},b_{i+1},a_{i+1},a_{i+2}$ are in convex position.
	The four sectors 
	$S(b_{i+2},a_{i-1},b_i,b_{i-1})$ 
	(red shaded area in Figure~\ref{fig:problem_2_2_page35}(\subref{fig:problem_2_2_page5})),
	$S(b_{i-1},b_i,a_i,b_{i+1})$ 
	(orange shaded area in Figure~\ref{fig:problem_2_2_page35}(\subref{fig:problem_2_2_page5})),
	$S(b_{i+1},a_i,\astar,a_{i+1})$ 
	(blue shaded area in Figure~\ref{fig:problem_2_2_page35}(\subref{fig:problem_2_2_page5})), and
	$S(a_{i+1},\astar,a_{i-1}',a_{i+2})$ 
	(gray shaded area in Figure~\ref{fig:problem_2_2_page35}(\subref{fig:problem_2_2_page5}))	
	contain the quadrilateral $\square(b_{i+2},b_{i+1},a_{i+1},a_{i+2})$ 
	(green area in Figure~\ref{fig:problem_2_2_page35}(\subref{fig:problem_2_2_page5})).
	The sectors are empty of points of~$P$ by Observation~\ref{observation:observation1}(\ref{observation:observation1_item1}).
	Consequently, the convex quadrilateral 
	$\square(b_{i+2},b_{i+1},a_{i+1},a_{i+2})$ is an $\ell$-divided 4-hole in~$P$.
	This concludes the proof of Claim~\ref{claim:lemma6_claim3}.
		
	To finish the proof of Lemma~\ref{lemma:lemma6}, recall that 
	all points of $B\setminus Q$ lie in \mbox{$\astar$-wedges} below~$W_{i+1}$ by Claim~\ref{claim_lemma6_claim_f}.
	Since $a_{i+2}$ is to the left of $a_{i+1}$,
	the line $\dline{a_{i+2}a_{i+1}}$ intersects $\ell$ above $\ell \cap W_{i+2}$.
	The line $\dline{a_{i+1}b_{i+1}}$ also intersects $\ell$ above $\ell \cap W_{i+2}$,
	since $a_{i+1}$ and $b_{i+1}$ both lie in~$W_{i+1}$.
	From $i=1$, every point of $B \setminus Q$ is 
	to the right of $\dline{a_{i+2}a_{i+1}}$ and
	to the right of $\dline{a_{i+1}b_{i+1}}$.
	Since the points $b_{i+2},b_{i+1},a_{i+1},a_{i+2}$ 
	form an $\ell$-divided 4-hole in~$P$ by Claim~\ref{claim:lemma6_claim3}, 
	Observation~\ref{observation:observation1}(\ref{observation:observation1_item1})
	implies that the sector $S(b_{i+2},b_{i+1},a_{i+1},a_{i+2})$ is empty of points of~$P$.
	Thus every point of $B \setminus Q$ lies to the left of $\dline{b_{i+1}b_{i+2}}$.
	Since $\overline{b_{i+1}b_{i+2}}$ intersects $\ell \cap W_{i+1}$ above $\ell \cap a_{i+1}b_{i+1}$ and 
	since $b_{i-1}$ lies to the left of $b_{i+2}$ and 
	to the left of $\dline{b_{i+1}b_{i+2}}$,
	every point of $B \setminus Q$ lies to the left of~$\dline{b_{i-1}b_{i+2}}$ and to the right of $b_{i+2}$,
	and thus in the sector $S(a_{i-1},b_i,b_{i-1},b_{i+2})$.
	However, by Observation~\ref{observation:observation1}(\ref{observation:observation1_item1}),
	this sector is empty of points of~$P$.
	Thus we obtain $B = \{b_{i-1},b_i,b_{i+1},b_{i+2}\}$,
	which contradicts the assumption $|B| \ge 5$. 
	This concludes the proof of Lemma~\ref{lemma:lemma6}.
\end{proof}

Next we show that if there is a sequence of consecutive \mbox{$a^*$-wedges} 
where the first and the last $a^*$-wedge both contain two points of~$B$
and every $a^*$-wedge in between them contains exactly one point of~$B$, 
then there is an $\ell$-divided 5-hole in~$P$.

\begin{lem}\label{lemma:lemma7}
Let $P = A \cup B$ be an $\ell$-divided set 
with $A$ not in convex position and with $|A| \ge 5$ and $|B| \ge 6$.
Let $W_i, \ldots, W_j$ be consecutive \mbox{$\astar$-wedges} 
with $1 \le i < j \le t$, $w_i=2=w_j$, 
and $w_k=1$ for every $k$ with $i<k<j$.
Then there is an $\ell$-divided 5-hole in~$P$. 
\end{lem}

\begin{proof}
For $i=j-1$, the statement follows by Lemma~\ref{lemma:lemma6}.
Thus we assume $j \ge i+2$.
That is, we have at least three consecutive \mbox{$a^*$-wedges}.
Suppose for contradiction that there is no $\ell$-divided 5-hole in~$P$.
Let $W\mathrel{\mathop:}=\bigcup_{k=i}^j W_k$ and $Q \mathrel{\mathop:}= P \cap W$. 
By Lemma~\ref{lemma:lemma5}, 
there is also no $\ell$-divided 5-hole in~$Q$.
Note that $|Q\cap B| = j-i+3$.
Also observe that $|Q \cap A|=j-i+2$ if $a_{i-1}=a_j=a_t$ 
and $|Q \cap A|=j-i+3$ otherwise.
We label the points in $B \cap W_i$ as $b_{i-1}$ and $b_i$ 
so that $b_{i-1}$ is to the right of $b_i$.
Further, we label the unique point in $B \cap W_k$ as $b_k$ for each $i < k < j$, and
the two points in $B\cap W_j$ as $b_{j}$ and $b_{j+1}$ so that $b_{j+1}$ is to the right of~$b_j$;
see Figure~\ref{fig:lemma3_illustration}.

\begin{figure}[htb]
	\centering
	\includegraphics{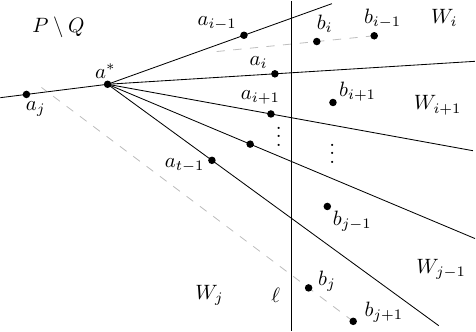}
	\caption{An illustration of \mbox{$\astar$-wedges} $W_i,\ldots,W_j$ in the proof of Lemma~\ref{lemma:lemma7}.}
	\label{fig:lemma3_illustration}
\end{figure}

\begin{claim}\label{claim:claim1}
All points of $B\cap (W_{k-1} \cup W_k \cup W_{k+1})$ are to the right of $\dline{a_{k}a_{k-1}}$ for every $k$ with $i < k < j$.
\end{claim}
The claim clearly holds for points from $B\cap W_k$. 
Thus it suffices to prove the claim only for points from 
$B\cup W_{k-1}$, 
as for points from $B\cup W_{k+1}$ it follows by symmetry.
Since $i<k<j$, 
Observation~\ref{observation:observation3}
implies that
the points $a_{k-1}$ and $a_k$ are both to the right of $\astar$.

We now distinguish the following two cases.

\begin{enumerate}
\item  
The point
$a_{k-2}$ is to the left of $\dline{\astar a_k}$;
see~Figure~\ref{fig:empty_lemma}(\subref{fig:empty_lemma1}).
Since $\astar$ is the rightmost inner point of~$A$,
$a_{k-1}$ does not lie inside the triangle $\triangle(\astar,a_k,a_{k-2})$
and thus $\square(a_{k-2},\astar,a_k,a_{k-1})$ is a 4-hole in~$P$.
All points of $B \cap W_{k-1}$ lie 
to the right of $\dline{\astar a_{k-2}}$
and to the left of $\dline{a_{k-2} a_{k-1}}$.
By Observation~\ref{observation:observation1}(\ref{observation:observation1_item2}), 
no point of $B \cap W_{k-1}$ lies in the sector $S(a_{k-2},\astar,a_k,a_{k-1})$ 
(red shaded area in Figure~\ref{fig:empty_lemma}(\subref{fig:empty_lemma1}))
and thus all points of $B\cap W_{k-1}$ are 
to the right of $\dline{a_ka_{k-1}}$.

\begin{figure}[htb]
	\hfill
	\begin{subfigure}[t]{.45\textwidth}\centering
		\includegraphics[page=1,width=\textwidth]{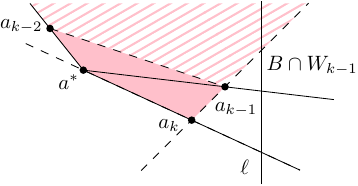}
		\caption{}
		\label{fig:empty_lemma1}
	\end{subfigure}
	\hfill
	\begin{subfigure}[t]{.45\textwidth}\centering
		\includegraphics[page=2,width=\textwidth]{empty_lemma}
		\caption{}
		\label{fig:empty_lemma2}
	\end{subfigure}
	\hfill
	\caption{An illustration of the proof of Claim~\ref{claim:claim1}.}
	\label{fig:empty_lemma}
\end{figure}

\item
The point
$a_{k-2}$ is to the right of $\dline{\astar a_k}$;
see~Figure~\ref{fig:empty_lemma}(\subref{fig:empty_lemma2}).
Since $a_{k-1}$ and $a_k$ are to the right of $\astar$
and since $a_{k-2}$ is to the left of $\dline{\astar a_{k-1}}$ 
and to the right of $\dline{\astar a_{k}}$,
the point $a_{k-2}$ is to the left of $\astar$. 
By Observation~\ref{observation:observation3}, 
we have $k=2$. 
That is, $W_{k-1}$ is the topmost $\astar$-wedge 
that intersects~$\ell$. 

There is another $\astar$-wedge below $W_{k+1}$,
since otherwise
$|B|=|B \cap (W_{k-1} \cup W_k \cup W_{k+1}) | \le 2+1+2 = 5$,
which is impossible according to the assumption $|B| \ge 6$.
By Observation~\ref{observation:observation3}, 
the point $a_{k+1}$ is to the right of $\astar$.
Moreover, since $\astar$ is the rightmost inner point of~$A$,
the point $a_{k}$ does not lie inside 
the triangle $\triangle(\astar,a_{k+1},a_{k-1})$.
The points $\astar,a_{k+1},a_k,a_{k-1}$ 
then form a 4-hole in~$P$, 
which has $\astar$ as the leftmost point.

By definition, all points of $B\cap W_{k-1}$ 
lie to the left of $\dline{\astar a_{k-1}}$.
As the ray $\overrightarrow{\astar a_{k+1} }$ intersects~$\ell$,
all points of $B\cap W_{k-1}$ lie also to the left of $\dline{\astar a_{k+1}}$.
By Observation~\ref{observation:observation1}(\ref{observation:observation1_item2}), 
no point of $B \cap W_{k-1}$ lies in the sector $S(\astar,a_{k+1},a_k,a_{k-1})$.
Thus all points of $B\cap W_{k-1}$ lie 
to the right of $\dline{a_ka_{k-1}}$.
\end{enumerate}
This finishes the proof of Claim~\ref{claim:claim1}.

We say that points $p_1,p_2,p_3,p_4$ form a \emph{counterclockwise-oriented convex quadrilateral} if every triple $(p_x,p_y,p_z)$ with $1 \leq x<y<z \leq 4$ is oriented counterclockwise.

\begin{claim}\label{claim:claim2}
The points $b_{i-1},b_i,a_i,a_{i+1}$ form a counterclockwise-oriented convex quadrilateral.
\end{claim}

Due to Claim~\ref{claim:claim1},
the points $b_{i-1}$ and $b_i$ are both to the right of $\dline{a_{i+1}a_i}$.
Thus the points $a_i$ and $a_{i+1}$ are both extremal points of those four points.
Also the point $b_{i-1}$ is extremal, 
since it is the rightmost of those four points.
The point $b_i$ does not lie inside the triangle $\triangle(a_{i+1},a_i,b_{i-1})$,
since, by Observation~\ref{observation:observation5},
$b_i$ lies to the left of $\dline{a_ib_{i-1}}$.
To finish the proof of Claim~\ref{claim:claim2},
it suffices to observe that the triples 
$(b_{i-1},b_i,a_i)$,
$(b_{i-1},b_i,a_{i+1})$,
$(b_{i-1},a_i,a_{i+1})$, and
$(b_i,a_i,a_{i+1})$
are all oriented counterclockwise.

\begin{claim}\label{claim:claim3}
The point $b_{i+1}$ lies to the right of $\dline{b_ib_{i-1}}$.
\end{claim}
	Suppose for contradiction 
	that $b_{i+1}$ lies to the left of $\dline{b_ib_{i-1}}$.
	We consider the five points $a_{i-1},a_i,b_{i-1},b_i,b_{i+1}$;
	see Figure~\ref{fig:lemma3_b3_specialcase}.
	By Claim~\ref{claim:claim1}, 
	the points $b_{i-1}$, $b_i$, and $b_{i+1}$ 
	lie to the right of $\dline{a_ia_{i-1}}$.
	Moreover, since $b_{i-1}$ and $b_i$ lie in $W_i$ and 
	since $b_{i+1}$ lies in $W_{i+1}$, 
	the points $b_{i-1}$ and $b_i$ both lie to the left of $\dline{a_ib_{i+1}}$.
	By Observation~\ref{observation:observation5},
	the point $a_{i-1}$ lies to the left of $\dline{b_ib_{i-1}}$
	and $b_{i+1}$ is to the right of $b_{i-1}$.
	Consequently, the points $b_{i-1}$ and $b_i$ 
	lie in the triangle $\triangle(a_{i-1},a_i,b_{i+1})$.
	Altogether, the points $a_{i-1}$, $b_i$, $b_{i-1}$, and $b_{i+1}$
	are in convex position. 
	
	\begin{figure}[htb]
	\centering
	\includegraphics{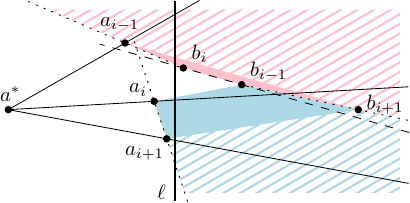}
	\caption{An illustration of the proof of Claim~\ref{claim:claim3}.}
	\label{fig:lemma3_b3_specialcase}
	\end{figure}
	
	By Claim~\ref{claim:claim1},
	the points $b_{i-1}$ and $b_{i+1}$ lie to the right of $\dline{a_{i+1}a_i}$. 
	Moreover, since $b_{i-1}$ is to the left of $b_{i+1}$ and to the left of $\dline{a_ib_{i+1}}$,
	the points $b_{i+1}$, $b_{i-1}$, $a_i$, and $a_{i+1}$ are in convex position.
	Since there are no further points in $W_i$ and $W_{i+1}$,
	the sets $\{a_{i-1},b_i,b_{i-1},b_{i+1}\}$ and $\{b_{i+1},b_{i-1},a_i,a_{i+1}\}$ are $\ell$-divided 4-holes in~$P$.
	By Observation~\ref{observation:observation1}(\ref{observation:observation1_item1}),
	the point $b_{i+2}$ lies neither in $S(a_{i-1},b_i,b_{i-1},b_{i+1})$ 
	nor in $S(b_{i+1},b_{i-1},a_i,a_{i+1})$.
	Recall that the ray $\overrightarrow{b_{i-1}b_{i+1}}$ intersects $\overrightarrow{\astar a_i}$
	and the ray $\overrightarrow{b_ia_{i-1}}$ does not intersect $\overrightarrow{\astar a_i}$.
	Therefore $b_{i+2}$ is to the right of $\dline{a_i a_{i+1}}$.
	This contradicts Claim~\ref{claim:claim1} 
	and finishes the proof of Claim~\ref{claim:claim3}.

\begin{claim}\label{claim:claim4}
For each $k$ with $i < k < j$, 
the point $b_k$ lies to the left of $\dline{a_k b_{i-1}}$ 
and to the left of~$b_{i-1}$.
\end{claim}
	Recall the labeling of the points in $B \cap W$; see Figure~\ref{fig:lemma3_illustration}.
	We show by induction on $k$ that
	\begin{enumerate}[(i)]
	 \item\label{lemma3_lem_t4_prop_1} 
	 the points $b_{i-1}$, $b_{k-1}$, $a_{k-1}$, and $a_k$
	 form a counterclockwise-oriented convex quadrilateral, 
	 which has $b_{i-1}$ as the rightmost point, and
	 \item\label{lemma3_lem_t4_prop_2} 
	 the point $b_k$ lies inside this convex quadrilateral 
	 and, in particular, to the left of $\dline{a_k b_{i-1}}$.
	\end{enumerate}
	Claim~\ref{claim:claim4} then clearly follows.
	
	For the base case,
	we consider $k=i+1$.
	By Claim~\ref{claim:claim2}, 
	the points $b_{i-1}$, $b_i$, $a_i$, and $a_{i+1}$ form a counterclockwise-oriented convex quadrilateral.
	By definition, $b_{i-1}$ is the rightmost of those four points.
	Figure~\ref{fig:lemma3_bi_end}(\subref{fig:lemma3_bi_end1}) gives an illustration.
	The point $b_{i+1}$ lies to the right of $\dline{a_{i+1}a_i}$
	and, by Claim~\ref{claim:claim3}, to the right of $\dline{b_ib_{i-1}}$.
	Moreover, since $b_{i+1}$ lies in $W_{i+1}$, it lies to the right of $\dline{a_ib_i}$.
	By Observation~\ref{observation:observation1}(\ref{observation:observation1_item1}),
	$b_{i+1}$ does not lie in the sector $S(b_{i-1},b_i,a_i,a_{i+1})$.
	Consequently,
	$b_{i+1}$ lies inside the quadrilateral $\square(b_{i-1},b_i,a_i,a_{i+1})$.
	
	\begin{figure}[htb]
		\hfill
		\begin{subfigure}[t]{.45\textwidth}\centering
			\includegraphics[page=1,width=\textwidth]{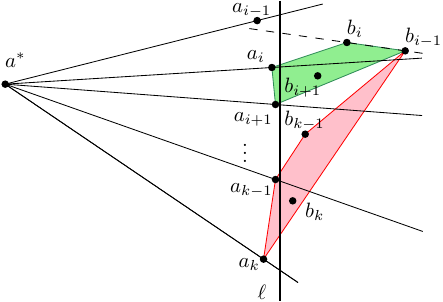}
			\caption{}
			\label{fig:lemma3_bi_end1}
		\end{subfigure}
		\hfill
		\begin{subfigure}[t]{.45\textwidth}\centering
			\includegraphics[page=2,width=\textwidth]{lemma3_bi_end}
			\caption{}
			\label{fig:lemma3_bi_end2}
		\end{subfigure}
		\hfill
		\caption{
		(\subref{fig:lemma3_bi_end1})~An illustration of the proof of Claim~\ref{claim:claim4}. 
		(\subref{fig:lemma3_bi_end2})~An illustration of the proof of Lemma~\ref{lemma:lemma7}.
		}
		\label{fig:lemma3_bi_end}
	\end{figure}

	For the inductive step,
	let $i+1 < k < j$.
	By the inductive assumption,
	the point $b_{k-1}$ lies to the left of $\dline{a_{k-1}b_{i-1}}$
	and to the left of $b_{i-1}$. 
	By Claim~\ref{claim:claim1}, $b_{k-1}$ lies to the right of $\dline{a_ka_{k-1}}$.
	Hence, the points $a_k$ and $b_{i-1}$ both lie 
	to the right of $\dline{a_{k-1}b_{k-1}}$.
	Recall that the points $b_{i-1},b_{k-1},a_{k-1},a_k$ lie to the right of $\astar$.
	Since $b_{i-1}$ is the first and 
	$a_k$ is the last in the clockwise order around $\astar$,
	the points $b_{i-1},b_{k-1},a_{k-1},a_k$ form 
	a counterclockwise-oriented convex quadrilateral,

	Recall that  
	the points $b_{k-1}$ and $b_k$ both 
	lie to the right of $\dline{a_ka_{k-1}}$ and that  
 	$b_{k-1}$ is to the left of $\dline{a_{k-1}b_{i-1}}$.
	Since $b_k \in W_k$, the point $b_k$ lies to the right of  $\dline{a_{k-1}b_{i-1}}$.
 	Therefore the clockwise order of $\{b_{k-1},b_{i-1},b_k\}$ around $a_{k-1}$ 
 	is $b_{k-1},b_{i-1},b_k$.
	Since $b_{i-1}$ is not contained in  $W_{k-1} \cup W_k$, the point 
 	$b_{i-1}$ is not contained in the triangle $\triangle(a_{k-1},b_k,b_{k-1})$.
	Consequently, 
	the points $a_{k-1},b_k,b_{i-1},b_{k-1}$ form a convex quadrilateral and,
	in particular, $b_k$ lies to the right of $\dline{b_{k-1}b_{i-1}}$.
	Figure~\ref{fig:lemma3_bi_end}(\subref{fig:lemma3_bi_end1}) gives an illustration.
	Since $b_k$ lies in $W_k$, 
	it lies to the right of $\dline{a_{k-1}b_{k-1}}$.
	By Observation~\ref{observation:observation1}(\ref{observation:observation1_item1}), the point $b_k$ does not lie in the sector $S(b_{i-1},b_{k-1},a_{k-1},a_k)$.
	Thus $b_k$ lies inside the quadrilateral $\square(b_{i-1},b_{k-1},a_{k-1},a_k)$.
	This finishes the proof of Claim~\ref{claim:claim4}.
	\medskip
	
	Using Claim~\ref{claim:claim4}, 
	we now finish the proof of Lemma~\ref{lemma:lemma7},
	by finding an $\ell$-divided 5-hole in the island $Q$ 
	and thus obtaining a contradiction with the assumption that there is no $\ell$-divided 5-hole in~$P$.
	In the following, 
	we assume, without loss of generality, 
	that $b_{j+1}$ is to the right of $b_{i-1}$.
	Otherwise we can consider a vertical reflection of~$P$.
	
	We consider the polygon $\mathcal{P}$ through the points 
	$b_{i-1},b_{j-1},a_{j-1},b_j,b_{j+1}$
	and we show that $\mathcal{P}$ is convex and empty of points of~$Q$.
	See Figure~\ref{fig:lemma3_bi_end}(\subref{fig:lemma3_bi_end2})
	for an illustration. 
	This will give us an $\ell$-divided 5-hole in~$Q$.
	
	We show that $\mathcal{P}$ is convex by proving that every point of 
	$\{b_{i-1},b_{j-1},a_{j-1},b_j,b_{j+1}\}$ 
	is a convex vertex of~$\mathcal{P}$.
	The point $a_{j-1}$ is a convex vertex of $\mathcal{P}$
	because it is the leftmost point in $\mathcal{P}$.
	The point $b_{i-1}$ is a convex vertex of $\mathcal{P}$
	because all points of $\mathcal{P}$ lie to the right of $\astar$
	and $b_{i-1}$ is the topmost point 
	in the clockwise order around $\astar$.
	The point $b_{j+1}$ is a convex vertex of $\mathcal{P}$
	because $b_{j+1}$ is the rightmost point in $\mathcal{P}$ 
	by Claim~\ref{claim:claim4} and by the assumption
	that $b_{j+1}$ is to the right of $b_{i-1}$.
	The point $b_{j-1}$ is a convex vertex of $\mathcal{P}$
	because $b_{j-1}$ lies to 
	the left of $\dline{a_{j-1}b_{i-1}}$ by Claim~\ref{claim:claim4}
	while $b_j$ and $b_{j+1}$ both lie to the right of this line.
	The point $b_j$ is a convex vertex of $\mathcal{P}$
	because, by Observation~\ref{observation:observation5}, 
	$b_j$ lies to the right of $\dline{a_{j-1}b_{j+1}}$
	while $b_{j-1}$ and $b_{i-1}$ both lie to the right of this line.
	Consequently, $\mathcal{P}$ is a convex pentagon 
	with vertices from both $A$ and~$B$. 
	Moreover, by Claim~\ref{claim:claim4}, all points $b_k$ with $i < k < j$ 
	lie to the left of $\overline{a_kb_{i-1}}$. 
	Since $b_i$ is to the left of $\dline{b_{j-1}b_{i-1}}$,  $\mathcal{P}$ is thus empty of points of~$Q$,
	which gives us a contradiction with the assumption that there is no $\ell$-divided 5-hole in~$P$.
\end{proof}

We now use Lemma~\ref{lemma:lemma7} 
to show the following upper bound on 
the total number of points of~$B$ 
in a sequence $W_i, \ldots, W_j$ of consecutive \mbox{$a^*$-wedges} with $w_i,\ldots,w_j \le 2$.

\begin{cor}\label{corollary:corollary8}
Let $P = A \cup B$ be an $\ell$-divided set 
with no $\ell$-divided 5-hole,
with $A$ not in convex position, and with $|A| \ge 5$ and $|B| \ge 6$.
For $1 \le i \le j \le t$, let $W_i, \ldots, W_j$ be consecutive \mbox{$\astar$-wedges} 
with $w_k \leq 2$ for every $k$ with $i \le k \le j$.
Then $\sum_{k=i}^j w_k \le j-i+2$.
\end{cor}

\begin{proof}
Let $n_0$, $n_1$, and $n_2$ be the number of \mbox{$a^*$-wedges} 
from $W_i,\dots,W_j$ with 0, 1, and 2 points of~$B$, respectively.
Due to Lemma~\ref{lemma:lemma7}, 
we can assume that
between any two \mbox{$a^*$-wedges} from $W_i,\dots,W_j$ with two points of $B$ each,
there is an $a^*$-wedge with no point of~$B$.
Thus $n_2 \le n_0 + 1$.
Since $n_0+n_1+n_2=j-i+1$,
we have
$
\sum_{k=i}^j w_k = 0 n_0 + 1 n_1 + 2 n_2 = (j-i+1) + (n_2 - n_0) \le j-i+2
$.
\end{proof}

\subsection{Computer-assisted results}
\label{sec:section_computer}

We now provide lemmas that are key ingredients in the proof of Theorem~\ref{theorem:theorem2}.
All these lemmas have computer-aided proofs.
Each result was verified by two independent implementations, 
which are also based on different abstractions of point sets; see below for details.

\begin{lem}\label{lemma:lemma9}
Let $P = A\cup B$ be an $\ell$-divided set 
with $|A|=5$, $|B|=6$, and with $A$ not in convex position.
Then 
there is an $\ell$-divided 5-hole in~$P$.
\end{lem}

\begin{lem}\label{lemma:lemma10}
Let $P = A\cup B$ be an $\ell$-divided set
with no $\ell$-divided $5$-hole in~$P$,
$|A| = 5$,  $4 \le |B| \le 6$, 
and with $A$ in convex position.
Then for every point $a$ of~$A$,
every convex $a$-wedge contains at most two points of~$B$.
\end{lem}

\begin{lem}\label{lemma:lemma11}
Let $P = A\cup B$ be an $\ell$-divided set
with no $\ell$-divided $5$-hole in~$P$,
$|A|=6$, and $|B|=5$.
Then for each point $a$ of~$A$,
every convex $a$-wedge contains at most two points of~$B$.
\end{lem}

\begin{lem}\label{lemma:lemma12}
Let $P = A\cup B$ be an $\ell$-divided set
with no $\ell$-divided $5$-hole in~$P$,
$5 \le |A| \le 6$, $|B| = 4$, and with $A$ in convex position.
Then for every point $a$ of~$A$,
if the non-convex $a$-wedge is empty of points of~$B$,
every $a$-wedge contains at most two points of~$B$.
\end{lem}

To prove these lemmas, we employ an exhaustive computer search through 
all combinatorially different sets of $|P| \leq 11$ points in the plane.
Since none of these statements depends on the actual coordinates of the points
but only on the relative positions of the points,
we distinguish point sets only by orientations of triples of points
as proposed by Goodman and Pollack~\cite{Goodman:1983:doi:10.1137/0212032}.
That is,
we check all possible equivalence classes of point sets in the plane
with respect to their triple-orientations, 
which are known as \emph{order types}.

We wrote two independent programs to verify Lemmas~\ref{lemma:lemma9} to \ref{lemma:lemma12}.
Both programs are available online~\cite{program_martin,program_manfred}.

The first implementation is based on programs from the two bachelor's theses of Scheu\-cher~\cite{scheucher2013,scheucher2014}.
For our verification purposes we reduced the framework from there to a very compact implementation~\cite{program_manfred}.
The program uses the order type database~\cite{aak-eotsp-01a,ak-aoten-06}, 
which stores all order types realizable as point sets of size up to 11. 
The order types realizable as sets of ten points 
are available online~\cite{aichDB} and
the ones realizable as sets of eleven points 
need about 96~GB and are available upon request from Aichholzer.
The running time of each of the programs in this implementation
does not exceed two hours on a standard computer.

The second implementation~\cite{program_martin} 
neither uses the order type database 
nor the program used to generate the database.
Instead it relies on the description of point sets 
by so-called \emph{signature functions}~\cite{bfk15,FW01_sweeps}.
In this description, points are sorted according to their $x$-coordinates and every unordered triple of points is represented by
a sign from $\{-,+\}$, where the sign is $-$ if the triple traced in the order by increasing $x$-coordinates is oriented clockwise and the sign is $+$ otherwise.
Every 4-tuple of points is then represented by four signs of its triples, which are ordered lexicographically. 
There are only eight 4-tuples of signs that we can obtain (out of 16 possible ones); see~{\cite[Theorem~3.2]{bfk15}} or~{\cite[Theorem~7]{FW01_sweeps}} for details.
In our algorithm, we generate all possible signature functions using a simple depth-first search algorithm and verify the conditions from our lemmas for every signature.
The running time of each of the programs in this implementation
takes up to a few hundreds of hours.

\subsection{Applications of the computer-assisted results}

Here we present some applications of the computer-assisted 
results from Section~\ref{sec:section_computer}.

\begin{lem}\label{lemma:lemma13}
Let $P = A \cup B$ be an $\ell$-divided set
with no $\ell$-divided $5$-hole in~$P$, with $|A| \ge 6$, and with $A$ not in convex position.
Then the following two conditions are satisfied.
\begin{enumerate}[(i)]
\item \label{lemma:lemma13_item1}
Let $W_{i},W_{i+1},W_{i+2}$ 
be three consecutive \mbox{$a^*$-wedges}
whose union is convex and
contains at least four points of~$B$.
Then $w_i,w_{i+1},w_{i+2} \leq 2$.
\item \label{lemma:lemma13_item2}
Let $W_{i},W_{i+1},W_{i+2},W_{i+3}$ 
be four consecutive \mbox{$a^*$-wedges} 
whose union is convex and 
contains at least four points of~$B$.
Then $w_i,w_{i+1},w_{i+2},w_{i+3} \leq 2$.
\end{enumerate}
\end{lem}

\begin{proof}
To show part~(\ref{lemma:lemma13_item1}), 
let $W \mathrel{\mathop:}= W_i \cup W_{i+1}\cup W_{i+2}$, 
$A' \mathrel{\mathop:}= A \cap W$, $B' \mathrel{\mathop:}= B \cap W$, 
and  $P' \mathrel{\mathop:}= A' \cup B'$.
Since $W$ is convex, 
$P'$ is an island of~$P$
and thus there is no $\ell$-divided 5-hole in~$P'$.
Note that $|A'| = 5$ and $A'$ is in convex position.
If $|B'| \leq 5$, 
then every convex $a^*$-wedge in~$P'$ 
contains at most two points of~$B'$
by Lemma~\ref{lemma:lemma10} applied to $P'$.
So assume that $|B'| \ge 6$.
If necessary, we remove points from $P'$ from the right to obtain $P''=A'\cup B''$, where
$B''$ contains exactly six points of~$B'$. 
Note that there is no $\ell$-divided 5-hole in $P''$,
since $P''$ is an island of~$P'$.
By Lemma~\ref{lemma:lemma10}, 
each $a^*$-wedge in $P''$ contains exactly two points of~$B''$.
Let $\tilde{B}$ be the set of points of $B$ that are to the left of the rightmost point of $B''$, including this point, and let $\tilde{P} \mathrel{\mathop:}= A \cup \tilde{B}$.
Note that $B'' \subseteq \tilde{B}$.
Since $|B''|=6$ and since $W \cap \tilde{B} = B''$,
each of the \mbox{$a^*$-wedges} $W_i,W_{i+1},W_{i+2}$ contains exactly two points of $\tilde{B}$.
The \mbox{$a^*$-wedges} $W_i$, $W_{i+1}$, and $W_{i+2}$ are also \mbox{$a^*$-wedges} in $\tilde{P}$.
Thus, Lemma~\ref{lemma:lemma6} applied to $\tilde{P}$ and $W_i,W_{i+1}$ then gives us an $\ell$-divided 5-hole in $\tilde{P}$.
From the choice of $\tilde{P}$, we then have an $\ell$-divided 5-hole in~$P$, a contradiction.

To show part~(\ref{lemma:lemma13_item2}), 
let $W \mathrel{\mathop:}= W_i \cup W_{i+1}\cup W_{i+2}\cup W_{i+3}$, 
$A' \mathrel{\mathop:}= A \cap W$, $B' \mathrel{\mathop:}= B \cap W$, 
and  $P' \mathrel{\mathop:}= A' \cup B'$.
Since $W$ is convex, 
$P'$ is an island of~$P$
and thus there is no $\ell$-divided 5-hole in~$P'$.
Note that $|A'| = 6$ and $A'$ is in convex position.
If $|B'| = 4$, 
then the statement follows from Lemma~\ref{lemma:lemma12} applied to $P'$
since $a^*$ is an extremal point of~$P'$.
If $|B'| = 5$, 
then the statement follows from Lemma~\ref{lemma:lemma11} applied to $P'$ and thus we can assume $|B'| \ge 6$.
Suppose for contradiction that $w_j \ge 3$ for some $i \le j \le i+3$. 
If necessary, we remove points from $P$ from the right to obtain $P''$
so that $B'' \mathrel{\mathop:}= P'' \cap B$ contains exactly six points of~$W \cap B$.
By applying part~(\ref{lemma:lemma13_item1}) for
$P''$ and $W_i \cup W_{i+1}\cup W_{i+2}$ and 
$W_{i+1} \cup  W_{i+2} \cup W_{i+3}$,
we obtain that 
$|B'' \cap W_i|, |B'' \cap W_{i+3}| = 3$ and 
$|B'' \cap W_{i+1}|, |B'' \cap W_{i+2}| = 0$.
Let $b$ be the rightmost point from $P'' \cap W$.
By Lemma~\ref{lemma:lemma11} applied to $W \cap (P''\setminus\{b\})$,
there are at most two points of $B'' \setminus \{b\}$
in every $a^*$-wedge in $W \cap (P'' \setminus \{b\})$.
This contradicts the fact that either
$|(B'' \cap W_i)\setminus \{b\}| = 3$ or $|(B'' \cap W_{i+3})\setminus \{b\}| = 3$.
\end{proof}

%\subsection{Extremal points of $\ell$-critical sets}
\subsection{Extremal points of \texorpdfstring{\boldmath{$\ell$}}{l}-critical sets}

Recall the definition of $\ell$-critical sets:
An $\ell$-divided point set $C=A \cup B$ is called $\ell$-critical if 
neither $C \cap A$ nor $C \cap B$ is in convex position and
if for every extremal point $x$ of~$C$, 
one of the sets $(C\setminus \{x\}) \cap A$ and 
$(C\setminus \{x\}) \cap B$ is in convex position.

In this section,
we consider an $\ell$-critical set $C=A \cup B$ with $|A|,|B| \ge 5$.
We first show that $C$ has at most two extremal points in~$A$ and 
at most two extremal points in~$B$.
Later, under the assumption
that there is no $\ell$-divided 5-hole in~$C$,
we show that $|B| \le |A|-1$ 
if $A$ contains two extremal points of~$C$ 
(Section~\ref{sec:two_extremal_points_in_A})
and that $|B| \le |A|$ 
if $B$ contains two extremal points of~$C$
(Section~\ref{sec:two_extremal_points_in_B}).

\begin{lem}\label{lemma:lemma4}
	Let $C = A \cup B$ be an $\ell$-critical set.
	Then the following statements are true.
\begin{enumerate}[(i)]
	\item \label{lemma:lemma4_item1}
	If $|A|\geq 5$, 
	then $|A \cap \partial\conv(C)| \le 2$.
		
	\item \label{lemma:lemma4_item3} 
	If $A \cap \partial\conv(C) = \{a,a'\}$, 
	then $a^*$ is the only inner point in $A$ 
	and every point of $A \setminus \{a,a'\}$ lies in the convex region
	spanned by the lines $\dline{\astar a}$ and $\dline{\astar a'}$ 
	that does not have any of $a$ and $a'$ on its boundary.
	
	\item \label{lemma:lemma4_item5} 
	If $A \cap \partial\conv(C) = \{a,a'\}$, 
	then the $a^*$-wedge that contains $a$ and $a'$ contains no point of~$B$.
\end{enumerate}
By symmetry, analogous statements hold for~$B$.
\end{lem}

\begin{proof}
To show statement~(\ref{lemma:lemma4_item1}),
suppose for contradiction that 
$|A \cap \partial\conv(C)| \ge 3$.
Let  $a$, $a'$, and $a''$ be three points from $A \cap \partial\conv(C)$ 
that are consecutive vertices of the convex hull $\conv(C)$.
If there is no point of $A$ in the triangle $\triangle(a,a',a'')$
spanned by the points $a$, $a'$, and $a''$,
then $A\setminus\{a'\}$ is not in convex position. 
This is impossible, since $C$ is an $\ell$-critical set.
If there is at least one point $a^{(1)}$ in $\triangle(a,a',a'')$, 
then we consider an arbitrary point $a^{(2)}$ from~$A\setminus\{a,a',a'',a^{(1)}\}$.
Such a point $a^{(2)}$ exists, since $|A|\ge 5$.
The point $a^{(1)}$ lies inside one of the triangles $\triangle(a,a',a^{(2)})$, $\triangle(a,a'',a^{(2)})$, 
or in $\triangle(a',a'',a^{(2)})$ 
and thus one of the sets $A\setminus\{a''\}$, $A \setminus\{a'\}$, or 
$A\setminus\{a\}$ is not in convex position,
which is again impossible.
In any case, $C$ cannot be $\ell$-critical 
and we obtain a contradiction. 

To show statement~(\ref{lemma:lemma4_item3}), 
assume that $A \cap \partial\conv(C) = \{a,a'\}$.
Every triangle in $A$ with a point of $A$ in its interior 
has $a$ and $a'$ as vertices,
as otherwise $A\setminus\{a\}$ or $A\setminus\{a'\}$ 
is not in convex position, which is impossible. 
Consider points $a^{(1)}$ and $a^{(2)}$ from $A$ 
such that $\triangle(a,a',a^{(1)})$ contains~$a^{(2)}$. 
Denote by $R$ the region bounded by $\overline{aa^{(2)}}$ and $\overline{a'a^{(2)}}$ that contains $a^{(1)}$.
If there is a point $a^{(3)}$ in $A \setminus (R \cup \{a,a'\})$ then 
$a^{(2)}$ lies in one of $\triangle(a,a^{(1)},a^{(3)})$ and $\triangle(a',a^{(1)},a^{(3)})$,
implying that $A\setminus\{a\}$ or $A\setminus\{a'\}$ 
is not in convex position. 
Hence all points of $A\setminus\{a,a',a^{(2)}\}$ lie in $R$.
Moreover, any further inner point $a^{(4)}$ from $A \cap R$ lies 
in some triangle $\triangle(a,a',a^{(5)})$ for some $a^{(5)} \in A \cap R$.
Thus, $a^{(4)}$ also lies in one of the triangles 
$\triangle(a,a^{(2)},a^{(5)})$ or $\triangle(a',a^{(2)},a^{(5)})$.
This implies that $A\setminus\{a\}$ or $A\setminus\{a'\}$ 
is not in convex position. 
Hence $a^{(2)}$ is the only inner point of~$A$.

To show statement~(\ref{lemma:lemma4_item5}), 
assume that $A \cap \partial\conv(C) = \{a,a'\}$.
Let $W_i$ be the wedge that contains $a$ and~$a'$.
Since $a$ and $a'$ are the only extremal points of $C$ contained in~$A$,
the segment $aa'$ is an edge of $\conv(C)$.
The points $a$, $a'$, and $a^* $ all lie in~$A$ and 
thus the triangle $\triangle(a,a',a^*)$ contains no points of~$B$.
Since all points of $C$ lie in the closed halfplane that is determined by the line $\overline{aa'}$ and that contains~$a^*$,
the wedge $W_i$ contains no points of~$B$.
\end{proof}

We remark that the assumption $|A| \ge 5$ 
in part (\ref{lemma:lemma4_item1}) 
of Lemma~\ref{lemma:lemma4} is necessary.
In fact, arbitrarily large $\ell$-critical sets 
with only four points in $A$ and 
with three points of $A$ on $\partial\conv(C)$ exist, 
and analogously for~$B$.
Figure~\ref{fig:minimal_sets}(\subref{fig:minimal_sets_C}) gives an illustration.

\begin{lem}\label{lemma:lemma14}
Let $C = A \cup B$ be an $\ell$-critical set
with no $\ell$-divided 5-hole in~$C$ and with $|A| \ge 6$.
Then $w_i \le 3$ for every $1 < i < t$.
Moreover, if $|A \cap \partial \conv (C)| = 2$,
then $w_1, w_t \le 3$.
\end{lem}

\begin{proof}
Recall that, since $C$ is $\ell$-critical, we have $|B| \ge 4$.
Let $i$ be an integer with $1 \le i \le t$.
First, assume that $A \cap \partial \conv (C) \subset W_i$. 
By Lemma~\ref{lemma:lemma4}(\ref{lemma:lemma4_item1}),
we have $|A \cap \partial \conv (C)| \in \{1,2\}$. 
If $|A \cap \partial \conv (C)|=1$,
then $i \in \{1,t\}$, so there is nothing to prove for~$w_i$. 
In the remaining case $|A \cap \partial \conv (C)|=2$,
by Lemma~\ref{lemma:lemma4}(\ref{lemma:lemma4_item5}) we have $W_i \cap B = \emptyset$, and thus $w_i=0$.

In the remaining case there is a point $a \in A \cap \partial \conv (C)\setminus W_i$. 
We consider $C' \mathrel{\mathop:}= C \setminus \{a\}$.
Since $C$ is an $\ell$-critical set, 
$A' \mathrel{\mathop:}= C' \cap A$ is in convex position.
Thus, there is a non-convex $\astar$-wedge $W'$ of~$C'$.
Since $W'$ is non-convex, all other \mbox{$\astar$-wedges} of~$C'$ are convex.
Moreover, since $W'$ is the union of the two
\mbox{$\astar$-wedges} of~$C$ that contain~$a$, 
all other \mbox{$\astar$-wedges} of~$C'$ are also \mbox{$\astar$-wedges} of~$C$.
Let $W$ be the union of all \mbox{$a^*$-wedges} of~$C$ 
that are not contained in $W'$.
Note that $W$ is convex and contains at least $|A|-3 \ge 3$ 
\mbox{$a^*$-wedges} of~$C$.
Since $|A| \ge 6$, the statement follows from Lemma~\ref{lemma:lemma13}(\ref{lemma:lemma13_item1}).
\end{proof}

%\subsubsection{Two extremal points of~$C$ in~$A$}
\subsubsection{Two extremal points of~\texorpdfstring{\boldmath{$C$}}{C} in~\texorpdfstring{\boldmath{$A$}}{A}}
\label{sec:two_extremal_points_in_A}

\begin{prop}\label{proposition:proposition15}
	Let $C = A \cup B$ be an $\ell$-critical set
	with no $\ell$-divided 5-hole in~$C$, with
	$|A|,|B|\ge 6$, and with 
	$|A \cap \partial \conv(C)| = 2$.
	Then $|B| \le |A|-1$.
\end{prop}

\begin{proof}
	Since $|A \cap \partial \conv(C)| = 2$,
	Lemma~\ref{lemma:lemma14} implies that 
	$w_i \le 3$ for every $1 \le i \le t$.
	Let $a$ and $a'$ be the two points in $A \cap \partial \conv(C)$.
	By Lemma~\ref{lemma:lemma4}(\ref{lemma:lemma4_item3}), 
	all points of $A \setminus \{a,a'\}$ lie in the convex region~$R$
	that is bounded by the lines $\dline{\astar a}$, $\dline{\astar a'}$, and $\ell$,
	and does not have any of $a$ and $a'$ on its boundary.
	That is, without loss of generality, $a=a_{h-1}$ and $a'=a_h$ for some $1 \le h \le |A|-1$
	and,
	by Lemma~\ref{lemma:lemma4}(\ref{lemma:lemma4_item5}),
	we have $w_h = 0$.
	Since all points of $A \setminus \{a,a'\}$ lie in the convex region~$R$,
	the regions 
	$W \mathrel{\mathop:}= \cl (\mathbb{R}^2 \setminus (W_{h-1} \cup W_h))$
	and 
	$W' \mathrel{\mathop:}= \cl (\mathbb{R}^2 \setminus (W_h \cup W_{h+1}))$ are convex; see~\figurename~\ref{fig:comment28}.
	Here $\cl(X)$ denotes the closure of a set~$X \subseteq \mathbb{R}^2$.
	Recall that the indices of the \mbox{$a^*$-wedges} are considered modulo $|A|-1$
	and that $\mathbb{R}^2$ is the union of all \mbox{$a^*$-wedges}.

	\begin{figure}[htb]
	\centering
		\includegraphics[page=2]{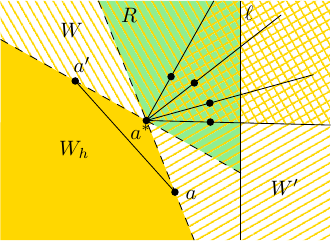}
	\caption{An illustration of the proof of Proposition~\ref{proposition:proposition15}.}
	\label{fig:comment28}
	\end{figure}

	First, suppose for contradiction
	that $|A|=6$.
	There are exactly five \mbox{$a^*$-wedges} $W_1,\ldots,W_5$,
	and only four of them can contain points of~$B$, since $w_h = 0$.
	We apply Lemma~\ref{lemma:lemma13}(\ref{lemma:lemma13_item1})
	to $W$ and to $W'$. An easy case analysis shows
	that either $w_i \le 2$ for every $1 \le i \le t$ 
	or $w_{h-1}, w_{h+1} = 3$ and $w_i=0$ for every $i \not \in \{h-1,h+1\}$.
	In the first case, 
	Corollary~\ref{corollary:corollary8} implies that $|B| \le 5$
	and in the latter case
	Lemma~\ref{lemma:lemma11} applied to $P \setminus \{b\}$, 
	where $b$ is the rightmost point of~$B$,
	gives $|B| \le 5$, a contradiction to $|B| \ge 6$.
	Hence, we assume $|A| \ge 7$.

\begin{claim}\label{claim:claim5}
For $1 \le k \le t-3$, if one of the four consecutive \mbox{$\astar$-wedges} 
$W_k$, $W_{k+1}$, $W_{k+2}$, or $W_{k+3}$ contains 3 points of~$B$, 
then $w_k+w_{k+1}+w_{k+2}+w_{k+3}=3$.
\end{claim}
	There are $|A|-1 \ge 6$ \mbox{$a^*$-wedges}
	and, in particular, 
	$W$ and $W'$ are both unions of at least four \mbox{$a^*$-wedges}.
	For every $W_i$ with $w_i = 3$ and 
	$1 \le i \le t$,
	the $a^*$-wedge $W_i$ is either contained in~$W$ or in~$W'$.
	Thus we can find four consecutive \mbox{$a^*$-wedges} 
	$W_k,W_{k+1},W_{k+2},W_{k+3}$ whose union is convex and contains $W_i$.
	Lemma~\ref{lemma:lemma13}(\ref{lemma:lemma13_item2}) implies
	that each of $W_k,W_{k+1},W_{k+2},W_{k+3}$ 
	except of~$W_i$ is empty of points of~$B$.
	This finishes the proof of Claim~\ref{claim:claim5}.

\begin{claim}\label{claim:claim6}
For all integers $i$ and $j$ with $1 \le i < j \le t$, 
we have $ \sum_{k=i}^j w_k \le j-i+2$.
\end{claim}

Let $S \mathrel{\mathop:}= (w_i,\ldots,w_j)$ 
and let $S'$ be the subsequence of $S$ obtained by removing every $1$-entry from~$S$. 
If $S$ contains only $1$-entries, the statement clearly follows.
Thus we can assume that $S'$ is non-empty.
Recall that, by Lemma~\ref{lemma:lemma14}, 
$S'$ contains only $0$-, $2$-, and $3$-entries, since $w_i \le 3$ for all $1 \le i \le t$. 
Due to Claim~\ref{claim:claim5}, 
there are at least three consecutive $0$-entries 
between every pair of nonzero entries of~$S'$ that contains a $3$-entry.
Together with Lemma~\ref{lemma:lemma7},
this implies that there is at least one $0$-entry between every pair of $2$-entries in~$S'$. 

By applying the following iterative procedure, 
we show that $\sum_{s \in S'} s \leq |S'| +1 $. 
While there are at least two nonzero entries in~$S'$, 
we remove the first nonzero entry $s$ from~$S'$.
If $s=2$, then we also remove the $0$-entry from $S'$ that succeeds~$s$ in~$S$.
If $s=3$, then we also remove the two consecutive $0$-entries from $S'$ that succeed~$s$ in~$S'$.
The procedure stops when there is at most one nonzero element $s'$ in the remaining subsequence $S''$ of~$S'$.
If $s'=3$, then $S''$ contains at least one $0$-entry and thus $S''$ contains at least $s'-1$ elements.
Since the number of removed elements equals the sum of the removed elements in every step of the procedure, 
we have  $\sum_{s \in S'} s \leq |S'| +1$. 
This implies
\[\sum_{k=i}^j w_k = \sum_{s \in S} s = |S|-|S'| + \sum_{s \in S'}s \le |S| - |S'| + |S'|+1 = j-i+2\]
and finishes the proof of Claim~\ref{claim:claim6}.

If $W_h$ does not intersect~$\ell$,
that is, $t < h \le |A|-1$,
then the statement follows from Claim~\ref{claim:claim6} applied with 
$i=1$ and $j=t$.
Otherwise, we have $h=1$ or $h=t$
and we apply Claim~\ref{claim:claim6} with
$(i,j)=(2,t)$ or $(i,j)=(1,t-1)$, respectively.
Since $t \leq |A|-1$ and $w_h = 0$, this gives us $|B| \le |A|-1$.
\end{proof}

%\subsubsection{Two extremal points of~$C$ in~$B$}
\subsubsection{Two extremal points of~\texorpdfstring{\boldmath{$C$}}{C} in~\texorpdfstring{\boldmath{$B$}}{B}}
\label{sec:two_extremal_points_in_B}

\begin{prop}\label{proposition:proposition16}
	Let $C = A \cup B$ be an $\ell$-critical set
	with no $\ell$-divided 5-hole in~$C$, with
	$|A|,|B|\ge 6$, and 
	with $|B \cap \partial \conv(C)| = 2$.
	Then $|B| \leq |A|$.
\end{prop}
	
\begin{proof}
	If $w_k \leq 2$ for all $1 \leq k \leq t$,
	then the statement follows 
	from Corollary~\ref{corollary:corollary8},
	since $|B| = \sum_{k=1}^t w_k \leq t + 1 \leq |A|$.
	Therefore we assume that there is an $\astar$-wedge $W_i$
	that contains at least three points of~$B$.
	Let $b_1$, $b_2$, and~$b_3$ be
	the three leftmost points in $W_i \cap B$ 
	from left to right.
	Without loss of generality,
	we assume that $b_3$ is to the left of $\dline{b_1b_2}$.
	Otherwise we can consider a vertical reflection of~$P$.
	Figure~\ref{fig:no3_first} gives an illustration.
	
	\begin{figure}[htb]
	\centering
	\includegraphics{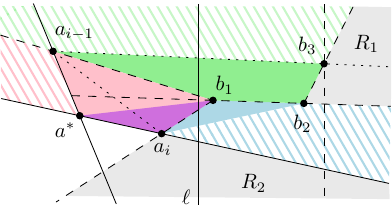}
	\caption{An illustration of the proof of Proposition~\ref{proposition:proposition16}.}
	\label{fig:no3_first}
	\end{figure}

	Let $R_1$ be the region 
	that lies
	to the left of $\overline{b_1b_2}$ 
	and to the right of $\overline{b_2b_3}$
	and let $R_2$ be the region 
	that lies 
	to the right of $\overline{a_ib_1}$ 
	and to the right of $\overline{a^*a_i}$.
	Let $B' \mathrel{\mathop:}= B \setminus \{b_1,b_2,b_3\}$.
	
	\begin{claim}\label{claim:claim7}
	Every point of $B'$
	lies in $R_1 \cup R_2$.
	\end{claim}
	We first show that 
	every point of $B'$ that 
	lies to the left of $\dline{b_1b_2}$ lies in $R_1$.
	Then we show that 
	every point of $B'$ that 
	lies to the right of $\dline{b_1b_2}$ lies in $R_2$.
	
	By Observation~\ref{observation:observation5},
	both lines $\dline{b_1b_2}$ and $\dline{b_1b_3}$ 
	intersect the segment $a_{i-1}a_i$.
	Since the segment $a_{i-1}b_1$ intersects~$\ell$ and since $b_1$ is the leftmost point of $W_i \cap B$,
	all points of~$B'$ that lie to the left of $\overline{b_1b_2}$
	lie to the left of $\overline{a_{i-1}b_1}$.
	The four points $a_{i-1},b_1,b_2,b_3$ 
	form an $\ell$-divided 4-hole in~$P$,
	since $a_{i-1}$ is the leftmost and 
	$b_3$ is the rightmost point of $a_{i-1},b_1,b_2,b_3$ 
	and both $a_{i-1}$ and $b_3$ lie to the left of $\overline{b_1b_2}$.
	By Observation~\ref{observation:observation1}(\ref{observation:observation1_item1}), 
	the sector  
	$S(a_{i-1},b_1,b_2,b_3)$ is empty of points of~$P$
	(green shaded area in Figure~\ref{fig:no3_first}).
	Altogether, 
	all points of~$B'$ that lie to the left of $\overline{b_1b_2}$
	are to the right of $\dline{b_2b_3}$ and thus lie in $R_1$.
	
	Since the segment $a_ib_1$
	intersects $\ell$ and since $b_1$ is the leftmost point of $W_i \cap B$,
	all points of~$B'$ that lie to the right of $\overline{b_1b_2}$
	lie to the right of $\overline{a_ib_1}$.
	By Observation~\ref{observation:observation1}(\ref{observation:observation1_item1}), 
	the sector $S(b_1,b_2,b_3,a_{i-1})$ is empty of points of~$P$.
	Combining this with the fact that $a^*$ is to the right of $\overline{a_{i-1}b_3}$, we see that $a^*$ lies to the right of $\dline{b_1b_2}$.
	Since $b_1$ and $b_2$ both lie to the left of $\dline{a^*a_i}$
	and since $a^*$ and $a_i$ both lie to the right of $\dline{b_1b_2}$,
	the points $b_2,b_1,a^*,a_i$ 
	form an $\ell$-divided 4-hole in~$P$.	
	By Observation~\ref{observation:observation1}(\ref{observation:observation1_item1}), 
	the sector
	$S(b_2,b_1,a^*,a_i)$ 
	(blue shaded area in Figure~\ref{fig:no3_first})
	is empty of points of~$P$.
	Altogether, 
	all points of~$B'$ that lie to the right of $\overline{b_1b_2}$
	are to the right of $\dline{a^*a_i}$ and to the right of $ \dline{a_ib_1}$
	and thus lie in $R_2$.
	This finishes the proof of Claim~\ref{claim:claim7}.
		
	\begin{claim}\label{claim:claim8}
	If $b_4$ is a point from $B' \setminus R_1$,
	then $b_2$ lies inside the triangle $\triangle( b_3,b_1, b_4)$.
	\end{claim}
	
	By Claim~\ref{claim:claim7},
	$b_4$ lies in $R_2$ and 
	thus to the right of $\dline{a_ib_1}$
	and to the right of $\dline{a^*a_i}$. 
	We recall that $b_4$ lies to the right of $\dline{b_1b_2}$.
	
	We distinguish two cases.
	First, we assume 
	that the points $b_2,b_3,b_1,a_i$ 
	are in convex position. 
	Then $b_2,b_3,b_1,a_i$ form an $\ell$-divided 4-hole in $P$
	and, by Observation~\ref{observation:observation1}(\ref{observation:observation1_item1}), 
	the sector $S(b_2,b_3,b_1,a_i)$ is empty of points from~$P$.
	Thus $b_4$ lies to the right of $\overline{b_2b_3}$
	and the statement follows.
	
	Second, we assume that the points $b_2,b_3,b_1,a_i$ 
	are not in convex position.
	Due to Observation~\ref{observation:observation5},
	$b_2$ and $b_3$ both lie to the right of $\dline{a_ib_1}$.
	Moreover, since $b_3$ is the rightmost of those four points,
	$b_2 $ lies inside the triangle $\triangle(b_3,b_1,a_i)$.
	In particular,
	$a_i$ lies to the right of $\overline{b_2b_3}$.
	Therefore, since $b_2$ and $b_3$ are to the left of $\dline{a^*a_i}$, 
	the line $\dline{b_2b_3}$ intersects $\ell$ in a point $p$ 
	above $\ell \cap \dline{a^*a_i}$.
	Let $q$ be the point $\ell \cap  \dline{b_1b_2}$.
	Note that $q$ is to the left of $\dline{a^*a_i}$.
	The point $b_4$ is to the right of $\overline{b_2b_3}$, 
	as otherwise $b_4$ lies in $\triangle(p,q,b_2)$, 
	which is impossible
	because the points $p,q,b_2$ are in $W_i$ while $b_4$ is not.
	Altogether,
	$b_2$ is inside $\triangle(b_3,b_1,b_4)$
	and this finishes the proof of Claim~\ref{claim:claim8}.
	
	\begin{claim}\label{claim:claim9}
	Either every point of $B'$ is to the right of $b_3$
	or $b_3$ is the rightmost point of~$B$.
	\end{claim}	
	
	By Observation~\ref{observation:observation1}(\ref{observation:observation1_item1}),
	the sector $S(b_3,a_{i-1},b_1,b_2)$ is empty of points of~$P$ and thus
	all points of $B' \cap R_1$ lie to the left of~$\dline{a_{i-1}b_3}$ and,
	in particular, to the right of~$b_3$.
	
	Suppose for contradiction that the claim is not true.
	That is, there is a point $b_4 \in B'$ that is the rightmost point in~$B$
	and there is a point $b_5 \in B'$ that is to the left of~$b_3$.
	Note that $b_4$ is an extremal point of~$C$.
	By Claim~\ref{claim:claim7} and by the fact that all points of $B' \cap R_1$
	lie to the right of~$b_3$,
	$b_5$ lies in $R_2 \setminus R_1$.	
	By Claim~\ref{claim:claim8}, 
	$b_2$ lies in the triangle $\triangle(b_1,b_5,b_3)$,
	and thus
	$B \setminus \{b_4\}$ is not in convex position.
	This contradicts the assumption that 
	$C$ is an $\ell$-critical set.
	This finishes the proof of Claim~\ref{claim:claim9}.
	
	\begin{claim}\label{claim:claim10}
	 The point $b_3$ is the third leftmost point of~$B$.
	 In particular, 
	 $W_i$ is the only $a^*$-wedge with at least three points of~$B$.
	\end{claim}
	
	Suppose for contradiction that $b_3$ is not the third leftmost point of~$B$.
	Then by Claim~\ref{claim:claim9}, $b_3$ is the rightmost point of $B$
	and therefore an extremal point of~$B$.
	This implies that $B'\subseteq R_2 \setminus R_1$, since all points of $B' \cap R_1$ lie to the right of~$b_3$.
	By Claim~\ref{claim:claim8}, each point of $B'$ then forms a non-convex quadrilateral
	together with $b_1$, $b_2$, and $b_3$.
	Since neither $b_1$ nor $b_2$ are extremal points of $C$ and since $|B \cap \partial\conv(C)|=2$, 
	there is a point $b_4 \in B$ that is an extremal point of~$C$.
	Since $|B| \ge 5$, the set $C \setminus \{b_4\}$ has none of its parts separated by $\ell$ in convex position,
	which contradicts the assumption that $C$ is an $\ell$-critical set.
	Since $W_i$ is an arbitrary $a^*$-wedge with $w_i \ge 3$,
	Claim~\ref{claim:claim10} follows.
	
	\begin{claim}\label{claim:claim11}	
	Let $W$ be a union of four consecutive \mbox{$\astar$-wedges} 
	that contains~$W_i$.
	Then \mbox{$|W \cap B| \le 4$}.
	\end{claim}
	
	Suppose for contradiction that $|W \cap B| \ge 5$.
	Let $C' \mathrel{\mathop:}= C \cap W$.
	Note that $|C' \cap A| = 6$ 
	and that $a^*,a_{i-1},a_i$ lie in $C'$.
	By Lemma~\ref{lemma:lemma5}, 
	there is no $\ell$-divided 5-hole in~$C'$.
	We obtain $C''$ by removing points from $C'$ 
	from the right, if necessary, until $|C'' \cap B| = 5$. 
	Since $C''$ is an island of~$C'$,
	there is no $\ell$-divided 5-hole in~$C''$.
	From Claim~\ref{claim:claim10} we know that
	$b_1,b_2,b_3$ are the three leftmost points in~$C$
	and thus lie in~$C''$.
	We apply Lemma~\ref{lemma:lemma11} to~$C''$
	and, since $b_1,b_2,b_3$ lie in a convex $a^*$-wedge of~$C''$,
	we obtain a contradiction.
	This finishes the proof of Claim~\ref{claim:claim11}.
	
	We now complete the proof of Proposition~\ref{proposition:proposition16}. 
	First, we assume that $1 \le i \le 4$.
	Let $W \mathrel{\mathop:}= W_1 \cup W_2 \cup W_3 \cup W_4$.
	By Claim~\ref{claim:claim11}, $|W \cap B| \le 4$.
	Claim~\ref{claim:claim10} implies that
	$w_k \le 2$ for every $k$ with $5 \le k \le t$.
	By Corollary~\ref{corollary:corollary8}, we have
	\[
	|B| = \sum_{k=1}^4 w_k + \sum_{k=5}^t w_k \le 4 + (t-3) = t+1 \le |A|.
	\]
	The case $t-3 \le i \le t$ follows by symmetry.

	Finally, we assume that $5 \le i \le t-4$.
	Let $W \mathrel{\mathop:}= W_{i-3} \cup W_{i-2} \cup W_{i-1} \cup W_i$.
	Note that $W$ is convex, since $2 \le i-3$ and $i <t$. 
	By Lemma~\ref{lemma:lemma13}(\ref{lemma:lemma13_item2}), we have $w_{i-3}+w_{i-2}+w_{i-1}+w_i \le 3$ and $w_i+w_{i+1}+w_{i+2}+w_{i+3} \le 3$.
	By Claim~\ref{claim:claim10}, 
	$w_k \le 2$ for all $k$ with $1 \le k \le i-4$.
	Thus, by Corollary~\ref{corollary:corollary8},
	$\sum_{k=1}^{i-4} w_k \le i - 3$.
	Similarly, we have $\sum_{k=i+4}^t w_k \le t - i - 2$.
	Altogether, we obtain that
	\[
	|B| 
	= \sum_{k=1}^{i-4} w_k 
	+ \sum_{k=i-3}^{i-1} w_k + w_i 
	+ \sum_{k=i+1}^{i+3} w_k 
	+ \sum_{k=i+4}^t w_k
	\leq (i-3)  + 3 + (t - i - 2) 
	= t-2 \leq |A| - 3.
	\]
\end{proof}

\subsection{Finalizing the proof of Theorem~\ref{theorem:theorem2}}
\label{section_final}

We are now ready to prove Theorem~\ref{theorem:theorem2}.
Namely, we show that for every $\ell$-divided set $P=A\cup B$ 
with $|A|, |B| \geq 5$ and with neither 
$A$ nor $B$ in convex position
there is an $\ell$-divided 5-hole in~$P$.

Suppose for the sake of contradiction 
that there is no $\ell$-divided 5-hole in~$P$.
By the result of Harborth~\cite{Ha78},
every set $P$ of ten points contains a 5-hole in~$P$.
In the case $|A|,|B|=5$, 
the statement then follows from the assumption that 
neither of $A$ and $B$ is in convex position.

So assume that at least one of the sets $A$ and $B$ has at least six points.
We obtain an island $Q$ of~$P$ 
by iteratively removing extremal points so that neither part is in convex position after the removal 
and until one of the following conditions holds.

\begin{enumerate}[(i)]
	\item\label{case:5_n} One of the parts $Q \cap A$ and $Q \cap B$ 
	has only five points. %, or
	\item\label{case:critical} $Q$ is an $\ell$-critical island of~$P$ 
	with $|Q \cap A|, |Q \cap B| \geq 6$.
\end{enumerate}

In case~(\ref{case:5_n}), 
we have $|Q \cap A|=5$ or $|Q \cap B|=5$.
We can assume by symmetry that 
$|Q \cap A|=5$ and $|Q \cap B| \ge 6$.
We let $Q'$ be the union of $Q \cap A$ 
with the six leftmost points of $Q \cap B$.
Since $Q \cap A$ is not in convex position, Lemma~\ref{lemma:lemma9} implies that
there is an $\ell$-divided 5-hole in~$Q'$, 
which is also an $\ell$-divided 5-hole in~$Q$,
since $Q'$ is an island of~$Q$.
However, this is impossible as then there is an $\ell$-divided 5-hole in~$P$ because $Q$ is an island of~$P$.

In case~(\ref{case:critical}),
we have $|Q \cap A|,|Q \cap B| \ge 6$.
There is no $\ell$-divided 5-hole in~$Q$, since $Q$ is an island of~$P$.
By Lemma~\ref{lemma:lemma4}(\ref{lemma:lemma4_item1}), 
we can assume without loss of generality that \mbox{$|A \cap \partial \conv(Q)| = 2$},
as $|A \cap \partial \conv(Q)|+|B \cap \partial \conv(Q)|\ge 3$ and 
thus $|A \cap \partial \conv(Q)|$ and $|B \cap \partial \conv(Q)|$ cannot be both smaller than~2.
Then it follows from Proposition~\ref{proposition:proposition15} 
that $|Q \cap B| < |Q \cap A|$.
By exchanging the roles of $Q \cap A$ and $Q \cap B$ 
and by applying Proposition~\ref{proposition:proposition16}, 
we obtain that $|Q \cap A| \le |Q \cap B|$, a contradiction.
This finishes the proof of Theorem~\ref{theorem:theorem2}.

%========================================================================================================
\section{Final Remarks}
\label{section_final_remarks}

At a first glance, 
it might seem that a similar approach could be used 
to derive stronger lower bounds also on the minimum number of 6-holes~$h_6(n)$.
However, since there are point sets of 29 points with no 6-hole~\cite{Overmars2002},
one would need to investigate point sets of size at least~30 
in order to find an $\ell$-divided 6-hole.
This task is too demanding for our implementations, 
since the number of combinatorially different point sets grows too rapidly.
Moreover, the case analysis in several steps of our proof would become much more complicated.

\iffalse
% ==== SOURCE FOR JOURNAL VERSION =====
In the statement of Theorem~\ref{theorem:theorem2} 
we require that the $\ell$-divided set $P = A \cup B$ 
satisfies $|A|,|B|\ge 5$.
In the arXiv version~\cite{arxiv_version}
we show that those requirements are necessary 
in order to guarantee an $\ell$-divided 5-hole in~$P$
by constructing 
an arbitrarily large $\ell$-critical set
$C = A \cup B$ with $|A|=4$ and with no $\ell$-divided 5-hole in~$C$.
Moreover, we prove the necessity of the assumptions in Lemmas~\ref{lemma:lemma9} to \ref{lemma:lemma12}.

\else
% ==== SOURCE FOR ARXIV VERSION =====
\subsection{Necessity of the assumptions in Theorem~\ref{theorem:theorem2}}
\label{section_necessity_of_theorem2}

In the statement of Theorem~\ref{theorem:theorem2} 
we require that the $\ell$-divided set $P = A \cup B$ 
satisfies $|A|,|B|\ge 5$.
We now show that 
those requirements are necessary 
in order to 
guarantee an $\ell$-divided 5-hole in~$P$
by constructing 
an arbitrarily large $\ell$-critical set
$C = A \cup B$ with $|A|=4$ and with no $\ell$-divided 5-hole in~$C$.

\begin{prop}\label{thm:carrotConstruction}
For every integer~$n\ge 5$,
there exists an $\ell$-critical set 
$C = A \cup B$ with $|A|=4$, $|B|=n$,
and with no $\ell$-divided 5-hole in~$C$.
\end{prop}

\begin{proof}
First, we consider the case where $n$ is odd.
Let $p^+ =(0,1)$ and $p^- = (0,-1)$ be two auxiliary points 
and let 
$\ell^+ = \{(x,y) \in \mathbb{R}^2 \colon y=x/4 \}$ 
and 
$\ell^- = \{(x,y) \in \mathbb{R}^2 \colon y=-x/4 \}$ 
be two auxiliary lines.
We place the point $b_1' = (2,-1/2)$ on the line $\ell^-$
and the auxiliary point $q = (2,1/2)$ on the line $\ell^+$.
For $i=2,\ldots,n$, 
we iteratively let
$b_i'$ be the intersection of the line $\ell^+ $ with the segment ${p^+ b_{i-1}' }$ if $i$ is even and
the intersection of $ \ell^-$ with ${p^- b_{i-1}'}$ if $i$ is odd.
We place two points $a_1$ and $a_2$ sufficiently close to $p^+$
so that $a_1$ is above $a_2$, 
the segment ${a_1a_2}$ is vertical with the midpoint $p^+$,
and all non-collinear triples $(b_i',b_j',p^+)$ 
have the same orientation as $(b_i',b_j',a_1)$ and $(b_i',b_j',a_2)$.
Similarly, we place two points $a_3$ and $a_4$ sufficiently close to $p^-$
so that $a_3$ is to the left of $a_4$, 
the segment $a_3a_4$ lies on the line $\overline{p^-q}$ and has $p^-$ as its midpoint,
the point $a_4$ is to the left of~$b_n'$,
and all non-collinear triples $(b_i',b_j',p^-)$ 
have the same orientation as $(b_i',b_j',a_3)$ and $(b_i',b_j',a_4)$.
Figure~\ref{fig:new_carrot} gives an illustration.

\begin{figure}[htb]
	\centering
	\includegraphics[width=0.7\textwidth]{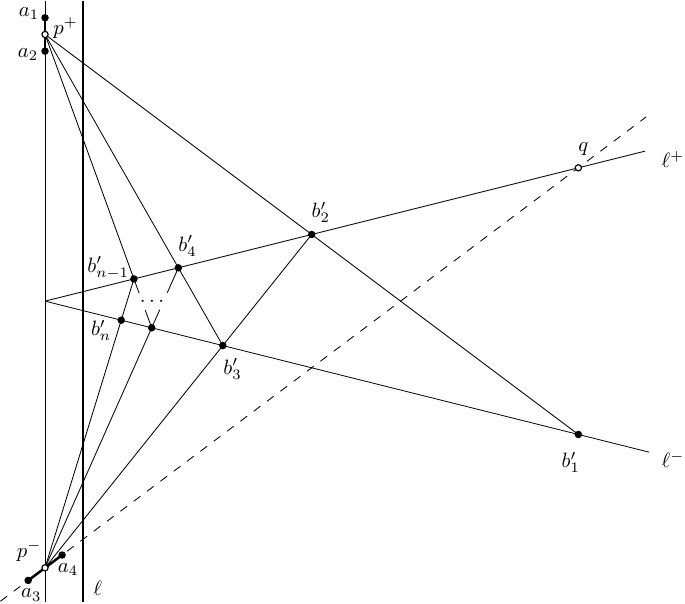}
	\caption{The set $C$ constructed in the proof of Proposition~\ref{thm:carrotConstruction} for $n$ odd.}
	\label{fig:new_carrot}
\end{figure}

We let~$A$, $B'$, and $B_3'$ be the sets $\{a_1,a_2,a_3,a_4\}$, $\{b_1' ,\ldots,b_n'\}$, and $B' \setminus \{b_3'\}$, respectively.
Note that the line $\overline{a_3a_4}$ intersects the segment $b_1'b_3'$.
Since $\max_{a \in A}x(a) < \min_{b' \in B'} x(b')$, 
the sets $A$ and $B'$ are separated by a vertical line~$\ell$.

Next we slightly perturb $b_3'$ to obtain a point $b_3$ such that 
$b_3$ lies above $\ell^-$ and all non-collinear triples $(b_3,c,d)$ with $c,d \in A \cup B_3'$ 
have the same orientation as $(b_3',c,d)$.
Note that the point $b_3$ lies in the interior of $\conv(B_3')$, since $n \ge 5$.

To ensure general position,
we transform 
every point $b_i'=(x,y) \in B_3' \cap \ell^+$ to $b_i=(x,y-\varepsilon x^2)$ and
every point $b_i'=(x,y) \in B_3' \cap \ell^-$ to $b_i=(x,y+\varepsilon x^2)$
for some $\varepsilon>0$.
The remaining points in $A \cup \{ b_3 \}$ remain unchanged. 
We choose $\varepsilon$ sufficiently small 
so that all non-collinear triples of points from $A \cup B_3' \cup \{b_3\}$ 
have the same orientations as their images after the perturbation.
Finally, let $B$ be the set $\{b_1,\ldots,b_n\}$ and set $B_3 \mathrel{\mathop:}= B \setminus \{b_3\}$.

Since the points from $B_3$ lie on two parabolas,
the set $B$ is in general position.
In particular,
points from $B_3$ are in convex position and
the point $b_3$ lies inside $\conv(B_3)$.
Also observe that the line $\ell$ separates $A$ and $B$
and that
$a_1$,
$a_3$,  and
$b_1$
are the extremal points of $C \mathrel{\mathop:}= A \cup B$.
Since neither of the sets $A$ and $B$ is in convex position,
and removal of any of the extremal points $a_1,a_3,b_1$ leaves either $A$ or $B$ in convex position,
the set $C = A \cup B$ is $\ell$-critical.

We now show that
$C$ contains no $\ell$-divided 5-hole.
Suppose for contradiction that there is an $\ell$-divided 5-hole $H$ in~$C$.
We set 
$A^+ \mathrel{\mathop:}= \{a_1,a_2\}$, 
$A^- \mathrel{\mathop:}= \{a_3,a_4\}$,
$B^+ \mathrel{\mathop:}= \{b_2,b_4,\ldots,b_{n-1}\}$, and 
$B^- \mathrel{\mathop:}= \{b_1,b_3,\ldots,b_n\}$.
First we assume that $H$ contains points from both $A^+$ and~$A^-$.
Then $H\cap B \subseteq \{b_{n-1},b_n\}$, since if there is a point $b_i$ in~$H$ 
with $i<n-1$, then $b_n$ lies in the interior of $\conv(H)$.
Note that if $H \cap B = \{b_{n-1},b_n\}$, then neither $a_4$ nor $a_1$ lies in~$H$ and thus $|H|<5$.
Hence $|H \cap B| = 1$, which is again impossible, as $H$ cannot contain all points from~$A$.
Therefore we either have $H \cap A \subseteq A^+$ or $H \cap A \subseteq A^-$ and, in particular, $1 \le |H \cap A|\le 2$.

We now distinguish the following two cases.

\begin{enumerate}
\item 
$|H \cap A| = 2$.
If $H\cap A=A^+$,
then the hole $H$ can contain only the point $b_n$ from~$B^-$.
This is because if there is a point $b_i $ in $ H \cap B^-$ with $i<n$,
then the point $b_{i+1}$ lies in the interior of~$\conv(H)$.
Additionally, $H$ contains at most two points from $B^+$,
since otherwise $H$ is not in convex position.
Consequently, $b_n$ lies in $H$ and $|H \cap B^+| = 2$, which is impossible, 
as $H$ would not be in convex position.

If $H\cap A=A^-$,
then the hole $H$ contains no point from~$B^+$.
This is because if there is a point $b_i $ in $ H \cap B^+$,
then the point $b_{i+1}$ lies in the interior of~$\conv(H)$.
The point $b_1$ cannot lie in $H$ because otherwise $H$ is not in convex position
as the line $\overline{a_3a_4}$ separates $b_1$ from $B \setminus \{b_1\}$.
Additionally, $H$ contains at most two points from~$B^-$,
since otherwise $H$ is not in convex position.
Thus $H$ contains at most four points of~$C$, which is impossible.

\item  
$|H \cap A| = 1$.
Assume first that $H \cap A \subseteq A^+$.
Note that for $b_i,b_j \in B^-$ with $i<j\le n$, 
the point $b_{i+1}$ lies inside the triangle $\triangle(a_1,b_i,b_j)$ and, if $j<n$,
the point $b_{j+1}$ lies inside $\triangle(a_2,b_i,b_j)$. 
Thus $H$ contains at most one point from~$B^-$ or we have $H \cap B^- = \{b_{n-2},b_n\}$ and $H \cap A = \{a_2\}$.
The latter case does not occur,
since for every $b_i \in B^+$ with $i<n-1$
the point $b_{n-1}$ lies in the interior of~$\conv(\{a_2,b_i,b_{n-2},b_n\})$.
Therefore we consider the case $|H \cap B^-| \le 1$.
However, $|H \cap B^+| \ge 3$ is impossible since $H$ would not be in convex position.
Altogether, we obtain $|H| <5$, which is impossible.

Now we assume that $H \cap A \subseteq A^-$.
Note that for $b_i,b_j \in B^+$ with $i<j<n$, 
the point $b_{i+1}$ lies inside the triangle $\triangle(a_4,b_i,b_j)$ and
the point $b_{j+1}$ lies inside $\triangle(a_3,b_i,b_j)$. 
Thus $H$ contains at most one point from~$B^+$.
Consequently, $H$ contains at least three points from $B^-$,
which is possible only if $H \cap B^- = \{b_1,b_3,b_5\}$.
However, then $H$ contains a point $b_i$ from $B^+$
and $b_3$ lies in the interior of $\conv (H)$.
\end{enumerate}
Thus, in any case, $H$ is not an $\ell$-divided 5-hole in~$C$,
a contradiction.

To finish the proof, we consider the case where $n$ is even. 
Let $\tilde{C} = A \cup \tilde{B}$ 
be the set constructed above
with $|A|=4$ and $|\tilde{B}|=n+1$.
We set $B \mathrel{\mathop:}= \tilde{B} \setminus \{b_2\}$ and $C \mathrel{\mathop:}= A \cup B$.
Note that $C$ is $\ell$-critical.

It remains to show that $C$ contains no $\ell$-divided 5-hole.
Suppose for contradiction that there is an $\ell$-divided 5-hole $H$ in~$C$. 
There is no $\ell$-divided 5-hole in $\tilde{C}$ and thus $b_2$ lies in the interior of $\conv(H)$.
Since $b_1$ is the only point from $C$ to the right of $b_2$,
the point $b_1$ lies in~$H$.
Since $a_1$ is the only point of $C$ to the left of $\overline{b_2b_1}$, 
all other points of $H$ lie to the right of $\overline{b_2b_1}$.
Then, however,
the set $(H \setminus \{a_1\}) \cup \{b_2\}$ is a 5-hole in $\tilde{C}$, which gives a contradiction.
\end{proof}

\subsection{Necessity of the assumptions in Lemmas~\ref{lemma:lemma9} to \ref{lemma:lemma12}}
\label{section_necessity_of_lemmas}

We remark that all the assumptions in the statements of Lemmas~\ref{lemma:lemma9} to \ref{lemma:lemma12} are necessary;
Figure~\ref{fig:examples}(\subref{fig:example_E}) shows that the conditions $|B|=5$ in Lemma~\ref{lemma:lemma11} and the convexity of $A$ in Lemma~\ref{lemma:lemma12} are both necessary.
The horizontal reflection of Figure~\ref{fig:examples}(\subref{fig:example_E}) also shows the necessity of the assumption $|A|=5$ in Lemma~\ref{lemma:lemma9}.
It follows from the example in Figure~\ref{fig:examples}(\subref{fig:example_G}) that the condition $|B|=4$ cannot be omitted in Lemma~\ref{lemma:lemma12}, since there is an $a$-wedge with three points of~$B$.
The same point set without the point $a'$ shows that the assumption $|B| \ge 4$ in Lemma~\ref{lemma:lemma10} is necessary. 
The example from Figure~\ref{fig:examples}(\subref{fig:example_J}) shows that the conditions $|B| = 6$ in Lemma~\ref{lemma:lemma9}, the convex position of $A$ in Lemma~\ref{lemma:lemma10}, and $|A|=6$ in Lemma~\ref{lemma:lemma11} are all necessary. 
The same set without the point $a$ shows that $|A|=5$ in Lemma~\ref{lemma:lemma10} is also needed
and, if we remove the points $a$ and $a'$, then the resulting point set shows that we need $5 \le |A|$ in Lemma~\ref{lemma:lemma12}.
We can make statements only about convex $a$-wedges in Lemmas~\ref{lemma:lemma10} and \ref{lemma:lemma11}, 
as there are counterexamples for the corresponding statements without the convexity condition.
It suffices to consider so-called \emph{double-chains}, which are point sets obtained by placing $n$ points on each of the two branches of a hyperbola.
Double-chains also show, that $A$ cannot be in convex position in Lemma~\ref{lemma:lemma9},
and, that the non-convex $a$-wedge must be empty of points in $B$ in Lemma~\ref{lemma:lemma12}. 

\begin{figure}[htb]
	\hfill
	\begin{subfigure}[t]{.32\textwidth}\centering
		\includegraphics[width=\textwidth]{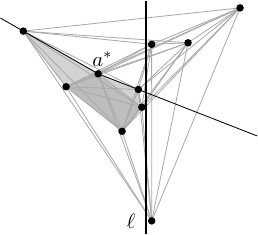}
		\caption{}
		\label{fig:example_E}
	\end{subfigure}
	\hfill
	\begin{subfigure}[t]{.32\textwidth}\centering
		\includegraphics[width=\textwidth]{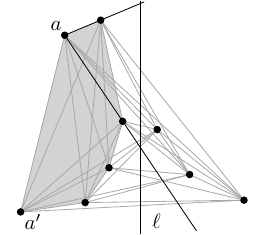}
		\caption{}
		\label{fig:example_G}
	\end{subfigure}
	\hfill
	\begin{subfigure}[t]{.32\textwidth}\centering
		\includegraphics[width=\textwidth]{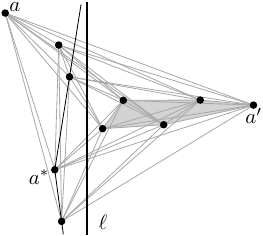}
		\caption{}
		\label{fig:example_J}
	\end{subfigure}
	\hfill
	\caption{
	Examples of points sets that witness tightness of Lemmas~\ref{lemma:lemma9} to \ref{lemma:lemma12}. All $k$-holes in these sets with $k \geq 5$ are highlighted in gray.
	}
	\label{fig:examples}
\end{figure}
%==== END OF SOURCE FOR ARXIV VERSION ====
\fi

%========================================================================================================
\section*{Acknowledgements}

The research for this article was partially carried out in the course of the bilateral research project ``Erd\H{o}s--Szekeres type questions for point sets'' 
between Graz and Prague, supported by the OEAD project CZ~18/2015 and project no.\ 7AMB15A~T023 of the Ministry of Education of the Czech Republic.

Aichholzer, Scheucher, and Vogtenhuber were partially supported by the ESF EUROCORES programme EuroGIGA -- CRP ComPoSe, Austrian Science Fund (FWF): I648-N18.
Parada was supported by the Austrian Science Fund (FWF): W1230.
Balko and Valtr were partially supported by the grant GAUK 690214.
Balko, Kyn\v{c}l, and Valtr were partially supported by the grant no. 18-19158S of the Czech Science Foundation (GA\v{C}R) and by the PRIMUS/17/SCI/3 project of Charles University.
%the project CE-ITI no. P202/12/G061 of the Czech Science Foundation (GA\v CR).
Balko and Kyn\v{c}l were partially supported by Charles University project UNCE/SCI/004.
Hackl and Scheucher were partially supported by the Austrian Science Fund (FWF): P23629-N18.
Balko, Scheucher, and Valtr were partially supported by the ERC Advanced Research Grant no 267165 (DISCONV).
%Scheucher was supported by DFG Grant FE~340/12-1.
Scheucher, Parada, and Vogtenhuber were partially supported within the collaborative DACH project \emph{Arrangements and Drawings}, by grants DFG: FE 340/12-1 and FWF: I 3340-N35, respectively.
%Parada, Scheucher, and Vogtenhuber were partially supported by the FWF project I 3340-N35 and the DFG grant FE 340/12-1 within the collaborative DACH project \emph{Arrangements and Drawings}." 

We thank G\"unter Rote and the anonymous reviewers for carefully going through the manuscript 
and for their valuable comments that helped to improve the quality of the paper and the overall presentation.

\bibliography{holesbib_arXiv}

\newpage
\appendix
\section{Flow summary}
\label{appendix:flow_summary}

\begin{figure}[htb]
  \centering
  \includegraphics{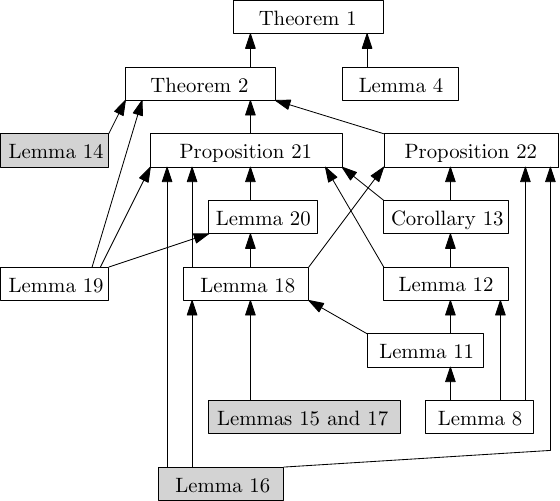}
  \caption{Flow summary. The shaded boxes correspond to computer-assisted results.}
\label{fig:flowDiagram}
\end{figure}
		
\end{document}